\documentclass[12pt,reqno]{amsart}
\usepackage{amsmath,amsxtra,amssymb,amsthm,amsfonts,eufrak,bm}
\usepackage{hyperref}
\usepackage{cleveref}

\usepackage{tikz-cd}
\makeatletter
\newcommand*{\rom}[1]{\expandafter\@slowromancap\romannumeral #1@}
\makeatother
\usetikzlibrary{matrix,arrows,decorations.pathmorphing}

\newcommand{\bx}{{\bf x}}

\newcommand{\lel}{\pl = \pl}

\newcommand{\R}{{\mathbb R}}

\newcommand{\Z}{{\mathbb Z}}
\newcommand{\C}{{\mathbb C}}
\newcommand{\ten}{\otimes}

\newcommand{\pl}{\hspace{.1cm}}

\newcommand{\ran}{\rangle}
\newcommand{\lan}{\langle}
\newcommand{\al}{\alpha}
\newcommand{\si}{\sigma}

\newcommand{\la}{\lambda}

\newcommand{\id}{\iota_{\infty,2}^n}

\newcommand{\F}{{\mathcal F}}
\newcommand{\E}{{\mathcal E}}

\newcommand{\A}{{\mathcal A}}
\newcommand{\B}{{\mathcal{B}}}
\newcommand{\D}{{\mathcal D}}
\newcommand{\M}{{\mathcal M}}

\newcommand{\T}{\mathbb{T}}
\renewcommand{\L}{\mathcal{L}}

\newcommand{\N}{{\mathcal N}}
\newcommand{\G}{\mathbb{G}}

\renewcommand{\o}[1]{\overset{\circ}{#1}}

\newcommand{\norm}[2]{\parallel \! #1 \! \parallel_{#2}}

\newtheorem{lemma}{Lemma}[section]
\newtheorem{prop}[lemma]{Proposition}
\newtheorem{theorem}[lemma]{Theorem}
\newtheorem{cor}[lemma]{Corollary}

\newtheorem{rem}[lemma]{Remark}

\newcommand{\re}{\begin{rem}\rm}
\newcommand{\mar}{\end{rem}}
\newtheorem{exam}[lemma]{Example}

\newcommand{\ket}[1]{|{#1}\rangle}

\newcommand{\qd}{\end{proof}\vspace{0.5ex}}
\newcommand{\prf}{\begin{proof}[\bf Proof:]}

\newcommand{\xspace}{\hbox{\kern-2.5pt}}

\renewcommand{\id}{\operatorname{id}}
\renewcommand{\D}{\Delta}
\renewcommand{\H}{\mathcal{H}}

\newcommand{\dom}{\operatorname{dom}}
\newcommand{\ARic}{\operatorname{GRic}}
\newcommand{\GRic}{\operatorname{GRic}}

\newcommand{\Ric}{\operatorname{GRic}}
\newcommand{\ric}{\operatorname{Ric}}
\renewcommand{\R}{\mathcal{R}}
\newcommand{\bten}{\overline{\otimes}}
\newcommand{\FF}{\mathbb{F}}

\newcommand{\tr}{\text{tr}}
\newcommand{\Ir}{\text{Irr}}

\renewcommand{\id} {\operatorname{id}}

\allowdisplaybreaks
\newtheorem{defi}[lemma]{Definition}
\newcommand{\vertiii}[1]{{\left\vert\kern-0.25ex\left\vert\kern-0.25ex\left\vert #1
    \right\vert\kern-0.25ex\right\vert\kern-0.25ex\right\vert}}
\renewcommand{\H}{\mathcal{H}}
\begin{document}
\title{Complete Logarithmic Sobolev inequality via Ricci curvature bounded below II}
%\title{Curvature approach for CLSI}
\author{Michael Brannan}
\address{Department of Mathematics\\ Texas A\&M University, College Station, TX 77840, USA} \email[Michael Brannan]{mbrannan@math.tamu.edu}
%\author{Li Gao}
\author{Li Gao}
\address{Department of Mathematics\\
Texas A\&M University, College Station, TX 77840, USA} \email[Li Gao]{ligao@math.tamu.edu}
\author{Marius Junge}
\address{Department of Mathematics\\
University of Illinois, Urbana, IL 61801, USA} \email[Marius Junge]{mjunge@illinois.edu}
\maketitle
\begin{abstract} Using a non-negative curvature condition, we prove the complete version of modified log-Sobolev inequalities for central Markov semigroups on various compact quantum groups, including group von Neumann algebras, free orthogonal group and quantum automorphism groups. We also prove that the ``geometric Ricci curvature lower bound'' introduced by Junge-Li-LaRacuente is stable under tensor products and amalgamated free products. As an application, we obtain the geometric Ricci curvature lower bound and complete modified logarithmic
Sobolev inequality for word-length semigroups on free group factors and amalgamated free product algebras.
\end{abstract}

\section{Introduction}
For Riemannnian manifolds, the Ricci curvature being bounded from below has many important applications in geometry and analysis. In recent years, progresses have been made to introduce a suitable notion of Ricci curvature lower bound for noncommutative spaces. Following the idea of $\Gamma$-calculus by Bakry-Emery, Junge and Zeng in \cite{JZ} studied the noncommutative curvature dimension condition, called $\Gamma_2$-condition. They proved that the $\Gamma_2$-condition, similar to the classical cases, implies $L_p$-Poincar\'e type inequalities, as well as its consequences, including deviation inequalities and transport inequalities. %They also showed the positive $\Gamma_2$-condition are satisfied by various examples.
On the other hand, using ideas from optimal transport theory, Carlen and Maas in \cite{CM} introduce a notion of Ricci curvature lower bound for quantum Markov semigroups. Their idea goes back to the famous works by Lott-Villani \cite{LV09} and Sturm \cite{sturm1}, which introduced a notion of Ricci curvature lower bound for metric measure spaces via certain convexity of entropy functionals.
This curvature condition, also called entropy Ricci curvature lower bound (in short, ERic), recently has attracted a lot of attention. It is proved in \cite{CM18} and \cite{DR} that for quantum Markov semigroups, ERic condition implies a modified log-Sobolev inequality, Talagrand's transport cost inequality and also an $L_2$-Poincar\'e inequality. All these give a unified picture of functional and geometric inequalities in both classical and noncommutative settings.

A common point in Junge-Zeng \cite{JZ} and Carlen-Maas \cite{CM} is to replace ``geometry'' by dynamics, described by a Makrov semigroup. This idea were used earlier by Erbar and Maas in \cite{ME11} to introduce the ERic condition for Markov semigroup on finite probability spaces. In the noncommutative setting, quantum Markov semigroups are generalizations of classical Markov semigroups, where the underlying probability space is replaced by matrix algebras or operators algebras. Quantum Markov semigroups have been widely used in operator algebras for the study of approximation properties, structure theory and noncommutative analysis (see e.g. \cite{Caspers18,jmp}). They often serve as a replacement of classical tools that are not available in the quantum setting. From this point of view, introducing curvature conditions for quantum Markov semigroups is very relevant.

The $\Gamma_2$ condition and the ERic condition turns out to be closely related. Indeed, both of them can be viewed as gradient estimates on certain weighted $L_2$-spaces. For the $\Gamma_2$ condition, the weight is given by a double operator integral of arithmetic means, while the ERic condition corresponds to logarithmic mean. For the heat semigroup on a Riemannian manifold, both conditions are equivalent to the lower bound of the Ricci curvature tensor. For quantum Markov semigroups, they can differ due to noncommutativity issues. More recently, motivated by the Bochner formula,  Li-Junge-LaRacuente introduce a notion of ``geometric Ricci curvature lower bound'' (in short, GRic). This GRic is a strong curvature condition that implies both $\Gamma_2$ and ERic, hence also the modified log-Sobolev inequality and its consequences.

%The GRic is a strong curvature condition that implies both $\Gamma_2$ and ERic, hence also modified log-Sobolev inequality and its consequences.

%It is likely that $\Gamma_2$ and ERic are different. This is suggested by that ERic admits a Bakry-Emery theorem that $\ERic\ge \la$ for a positive $\la$ implies $\la$-modified log sobolev inequality, whereas such an implication is not clear for $\Gamma_2$. More recently, Li, Junge and LaRacuente motivated from Bochner formula introduce a notion of ``Geometric Ricci curvature lower bound'' (in short, GRic). The GRic is a strong curvature condition that implies both $\Gamma_2$ and ERic, hence also modified log-Sobolev inequality and its consequences.

In this paper, which is the second in a series of two papers (see \cite{BGJ}), we continue our study of the Ricci curvature condition and its connection to the complete version of modified log-Sobolev inequality (in short, CLSI). We focus on various concrete examples in operator algebras and prove the following results:
\begin{enumerate}
\item[i)] Central Markov semigroups on compact quantum groups always have $\GRic\ge 0$. Based on this, we show that
  under certain growth condition for the length functions, Fourier multiplier semigroups on group von Neumann algebras have positive CLSI constant.
    We also prove that the heat semigroups on free orthogonal group and quantum isomorphism groups (tracial case) has CLSI.
  \item[ii)] The GRic condition is stable under tensor product and free product.
  \item[iii)] The word-length semigroup on $q$-Gaussian and free group factors satisfy sharp $\ARic\ge 1$ and $1$-CLSI.
   \item[iv)] The generalized depolarizing semigroup has $\GRic\ge 1/2$.
   \item[v)] Curvature lower bounds and positive CLSI constants for some natural semigroups on quantum tori.
\end{enumerate}
There are two ingredients in our proof. The first one is the interwining relation from \cite{CM}
which we use to prove our curvature condition. The second tool is the main result of our first paper \cite{BGJ}, which enables us to obtain complete log-Sobolev inequality from non-negative curvature condition. This is essential for the examples in i) because the curvature lower bound is not strictly positive and the Bakry-Emery type theorem does not apply.

The paper is organized as follows. Section 2 reviews the definitions and previous results that will be used in the rest of paper. In Section 3, we discuss central Markov semigroups on compact quantum groups, including group von Neumann algebras, free orthogonal group and quantum automorphisms groups. We prove in Section 4 that GRic is stable under tensor product and amalgamated free product. Section 5 is devoted to optimal GRic constant for word length semigroup on $q$-Gaussian and free group factors. Section 6 revisit generalized depolarizing semigroup and some semigroups on quantum tori.

\subsection*{Note added}
While this manuscript was being prepared, the authors learned that several of the examples studied in this work were independently considered in the recent work of Wirth and Zhang \cite{WZ} in the context of a complete version of the gradient estimate (corresponding to ERic). They obtained the complete gradient estimate results parallel to our study of GRic in this paper.  

%During the preparation of this work, we got to know that several examples above have been independently studied by Wirth and Zhang \cite{WZ} for complete version of gradient estimate that corresponds to ERic. They obtained the complete gradient estimate results parallel to our study of GRic.

% that studies the complete version of gradient estimate that corresponds to ERic. They independently prove the complete gradient estimate for i)-iv), parallel to our study of GRic. From their complete gradient estimate, the sharp CLSI constants in iv) are also obtained. We emphasis that our argument are different due to the difference of definitions in curvature condition. Moreover, it is known that GRic implies the complete gradient estimate they considered.

\subsection*{Acknowledgements}
 Michael Brannan was partially supported by NSF Grants DMS-2000331 and  DMS-1700267.   Marius Junge was partially supported by NSF grants DMS-1839177 and  DMS-1800872.
 Parts of this work were completed at the 2019 Great Plains Operator Theory Symposium at Texas A\&M University and the QLA Meets QIT 2019 conference at Purdue University.  The authors thank the organizers of these conferences for the stimulating work environment.  The authors also thank the organizers of the 48th Canadian Operator Symposium, where a preliminary version of these results were presented.

\section{Preliminary}
\subsection{Quantum Markov semigroups}Throughout the paper, $\M$ will always be a finite von Neumann algebra equipped with a normal faithful tracial state $\tau$. For $0<p<\infty$, the $L_p$-space $L_p(\M)$ is defined as the completion of $\M$ with respect to the norm
$\norm{a}{p}=\tau(|a|^{p})^{1/p}$. We identify $L_\infty(\M):=\M$ and $L_1(\M)\cong\M_*$.

A quantum Markov semigroup is a family of linear maps $\displaystyle (T_t)_{t\ge 0}:\M\to \M$ satisfying the following properties
\begin{enumerate}
\item[i)] $T_t$ is a normal unital completely positive map for all $t\ge 0$.
\item[ii)] $T_t\circ T_s=T_{s+t}$ for any $t,s\ge 0$ and $T_0=\id$.
\item[iii)] for each $x\in \M$, $t\mapsto T_t(x)$ is continuous in ultra-weak topology.
\end{enumerate}
We say a quantum Markov semigroups $(T_t)$  is \emph{symmetric} if
for any $t$, %$T_t$ is a self-adjoint map for the $\tau$-inner product,
\[\tau(x^*T_t(y))=\tau(T_t(x)^*y)\pl , \pl  x,y\in \M.\]
For symmetric $(T_t)$,
the fixed-point subspace ${\N=\{x\in \M\pl | \pl T_t(x)=x, \forall t\ge 0\}}$ forms a subalgebra and each $T_t$ is an $\N$-bimodule map,
\[T_t(axb)=aT_t(x)b\pl, \pl  \forall\pl a,b\in \N ,x\in \M\]
In particular, we have
\[T_t\circ E= E\circ T_t=E\pl.\]
where $E:\M\to \N$ is the trace preserving conditional expectation onto $\N$ given by
\[ \tau(xy)=\tau(xE(y)), \forall x\in \N,y\in \M\pl. \]
We say $(T_t)$ is \emph{ergodic} if $\N=\C 1$ is trivial. This means the semigroup admits an unique invariant state as the identity element $1$. Throughout the paper, we will focus on symmetric quantum Markov semigroups that are not necessarily ergodic. We refer to \cite{DL92} for more information of symmetric quantum Markov semigroups.

Denote the generator of the semigroup as
\[\pl Ax=w^*\text{-}\lim_{t\to 0} \frac{x-T_t(x)}{t}\pl, \pl T_t=e^{-At}\pl,\]
For symmetric semigroups, $T_t=T_t^\dag$ are unital completely positive and trace preserving , and the generator $A$ is a self-adjoint and positive operator on $L_2(\M)$.
A symmetric quantum Markov semigroup is determined by its
{\it  Dirichlet form} \[\E:L_2(\M)\to [0,\infty]\pl ,\pl  \E(x,x)=\tau(x^*A x)\pl.\]
We write $\dom (A)$ for the domain of $A$ and $\dom (A^{1/2})$ for the domain of $\E$. The Dirichlet subalgebra $\A_\E:=\dom (A^{1/2})\cap \M$ is a dense $*$-subalgebra of $\M$ and a core of $A^{1/2}$ \cite{DL92}. $\A_\E$ is a core for $\E$ (or $A^{1/2}$), i.e. closed under the graph norm $\norm{x}{\E}=\norm{x}{2}+\norm{A^{1/2}x}{2}$. In particular, we have $A(\N)=0$ and $\N\subset \A_\E$.

\subsection{Modified logarithmic Sobolev inequalities}
Let $(\M,\tau)$ be a finite von Neumann algebra. We say $\rho \in L_1(\M)$ is a density operator (or simply density) if $\rho\ge 0$ and $\tau(\rho)=1$.
Using the identification $\M_*\cong L_1(\M)$ via duality
\[a\in L_1(\M)\longleftrightarrow \phi_a\in \M_*,\pl  \phi_a(x)=\tau(ax)\pl,\]
density operators corresponds to the normal states of $\M$. Throughout the paper, states always mean normal states and are identified with their density operators. We write $S(\M)=\{\rho\in L_1(\M)\pl |\pl \rho\ge 0, \tau(\rho)=1 \}$ as the state space of $\M$ and for a subalgebra $\A\subset \M$, we write $S(\A):=S(\M)\cap \A$ as the states with bounded density operators in $\A$.

Recall that for two invertible densities $\rho$ and $\si$, the relative entropy is given by
\[D(\rho||\si)=\tau(\rho \log \rho-\rho\log \si)\pl,\]
and for general states, $\displaystyle D(\rho||\si):=\lim_{\epsilon\to 0}D(\rho+\epsilon 1||\si+\epsilon 1)$.
Let $\N\subset \M$ be a von Neumann subalgebra and $E:\M\to \N$ be the trace preserving conditional expectation on to $\N$. For a state $\rho$, the relative entropy with respect to $\N$ is defined as follows
\[D(\rho||\N):=\inf_{\si\in S(\N)} D(\rho||\si)=D(\rho||E(\rho))\pl.\]
where the infimum is always attained by $E(\rho)$. In the case $\N=\mathbb{C}1$, we write
$H(\rho):=D(\rho||1)$ as the entropy of $\rho$. (Note that here $H$ differs with the usual von Neumann entropy in information theory due to the normalization of the trace).

Let $T_t=e^{-At}:\M\to\M$ be a quantum Markov semigroup with generator $A$. The Fisher information of a state $\rho\in S(\A_\E)$ is given by
\[ I(\rho)=\tau(A\rho \log \rho )=\E(\rho,\log\rho)\pl.\]
\begin{defi}i) We say $T_t$ satisfies $\la$-modified logarithmic Sobolev inequality ($\la$-MLSI) for $\la >0$ if for any $\rho\in S(\A_\E)$,
\[ 2\la D(\rho||\N)\le I(\rho)\pl. \]
ii) We say $T_t$ satisfies $\la$-complete logarithmic Sobolev inequality ($\la$-CLSI) for $\la >0$ if for all finite von Neumann algebra $\R$, $\id_\R\ten T_t$ satisfies $\la$-MLSI.
\end{defi}

Let us recall that for ergodic $T_t$, $T_t$ satisfies $\la$-logarithmic Sobolev inequality ($\la$-LSI) for $\la >0$ if for any $\rho\in S(\A_\E)$,
\[\la H(\rho)\le 2\E(\rho^{1/2},\rho^{1/2})\pl.\]
It was proved in \cite{HC} that $\la$-LSI is equivalent to
the hypercontractivity that $\norm{T_t:L_2(\M)\to L_p(\M)}{}\le 1$ if $p\le 1+e^{2\la t}$, and for symmetric semigroup $\la$-LSI implies $\la$-MLSI by $L_p$-regularity.

\subsection{Curvature conditions}
We now review different curvature conditions for quantum Markov semigroups. Recall that the {\it gradient form} (or \emph{carr\'e du champ}) of the generator $A$ is given by
\begin{align}
\label{gd} \Gamma(x,y):=\frac{1}{2}\Big((Ax^*)y+x^*Ay-A(x^*y)\Big)\pl, \forall\pl x,y \in \dom (A)\end{align}
and it be extended to $x,y\in \dom(A^{1/2})$. We recall the following concept of derivation triple from \cite{JLLR, BGJ}. Recall that $\A_\E=\M\cap \dom(A^{1/2})$ is the Dirichlet algebra.

%Let $\hat{\M}$ be a finite von Neumann algebra $\hat{\M}$ containing $(\M,\tau)$ with induced trace. The trace of $\hat{\M}$ is also denoted by $\tau$. We say a linear map $\delta:{\rm dom}(A^{1/2})\to L_2(\hat{\M})$ is a symmetric derivation if $\delta(x^*)=\delta(x)^*$ and  $\delta$ satisfies the Leibniz rule:
%\[\delta(ab)=\delta(a)b+a\delta(b)\pl, \pl \forall a,b \in \in \A_\E \pl.
%\]
\begin{defi}\label{drivationtriple}Let $T_t:\M\to \M$ be a symmetric quantum Markov semigroup. A derivation triple of $(\A,\hat{\M},\delta)$ of $T_t$ consists of
\begin{enumerate}
\item[i)] a weak$^*$-dense subalgebra $\A\subset \M$ such that $T_t(\A)\subset \A$ and $\A \subset \A_\E$. %is core of $A^{1/2}$.
\item[ii)] a finite von Neumann algebra $\hat{\M}$ such that $\M\subset \hat{\M}$ with induced trace.
\item[iii)] a symmetric derivation $\delta:\A\to L_2(\hat{\M})$, meaning that
    $\delta(x^*)=\delta(x)^*$ and $\delta$ satisfies the Leibniz rule:
\begin{align*}\delta(xy)=\delta(x)y+x\delta(y)\pl, \pl \forall x,y \in \A \pl.
\end{align*}
Moreover, $\delta$ and $\Gamma$ are related through \begin{align}E_\M(\delta(x)^*\delta(y))=\Gamma(x,y)\pl, \pl x,y\in \A \label{deltagamma}\end{align} where $E_\M:\hat{\M} \to \M$ is the conditional expectation.
\end{enumerate}
We say the derivation $\delta$ has mean zero property if $E_\M(\delta(x))=0$ for all $x\in \A$.
\end{defi}
A consequence of i) and iii) is that $A=\delta^*\overline{\delta}$ for the closure $\overline{\delta}$ on $\dom(A^{1/2})$.
It was proved in the unpublished preprint \cite{JRS} that $T_t$ admits a derivation triple with $\A=\A_\E$ if and only if $T_t$ satisfies $\Gamma$-regularity that for all $x\in \dom(A^{1/2})$, $\Gamma(x,x)\in L_1(\M)$. Nevertheless,
it is sufficient and often more convenient to work with subalgebra $\A\subset \A_\E$ with stronger regularity.

Derivation triple is the key concept ingredient in the following definition of geometric curvature lower bound. For an subalgebra $\A\subset\M$, we denote $\A_0=\cup_{t>0}T_t(\A)$. $\A_0\subset \dom(A)$ and $A(\A_0)\subset \dom(A^{1/2})$ because $AT_t$ and $A^{3/2}T_t$ are bounded operator on $L_2(\M)$. We also write $\Omega_\delta=\overline{\A\delta(\A)}\subset L_2(\hat{\M})$ as the $\A$-bimodule generated by the range of $\delta$. %Note that by density, $\Omega_\delta=\overline{\A_\E\delta(\A_\E)}$ is a $\A_\E$-bimodule.
\begin{defi}\label{defi}
We say $(\A,\hat{\M},\delta)$ satisfies a geometric Ricci curvature lower bound $\la$ for $\la\in \mathbb{R}$ (in short $\ARic\ge \la $) if there exists a symmetric quantum Markov semigroup $\hat{T}_t=e^{-\hat A{t}}:\hat{\M}\to \hat{\M}$ with generator $\hat{A}$ such that
\begin{enumerate}
\item[i)] $\hat{T}_t|_\M=T_t$ for any $t\ge 0$.
\item[ii)] $\delta(\A_0)\subset \dom(\hat{A})$ and there exists a $\A$-bimodule operator $\ric: \Omega_\delta\to L_2(\hat{\M})$ such that for $x\in \A_0$,    \begin{align}\label{inter}\ric(\delta(x))=\delta A(x)-\hat{A}\delta(x). \end{align}
\item[iii)] for any $y\in \Omega_\delta$, \begin{align}\lan y, \ric(y)\ran\ge \la \lan y, y\ran\pl.\label{L2}\end{align}
    where $\lan \cdot,\cdot\ran$ is the trace inner product of $(\hat{\M},\tau)$.
\end{enumerate}
\end{defi}
The motivation of above definition is of course the Bochner–Weitzenb\"ock–Lichnerowicz formula (c.f. pp374 \cite{OT})
\[\Delta (\nabla f)-\nabla (\Delta f)+\ric(\nabla f)=0 \pl,\]
where $\Delta=\nabla^*\nabla$ is the Laplace-Beltrami operator on a Riemannian manifold and $\nabla$ is the gradient operator. A special case that repeatedly occurs in our examples is $\ric=\la\id$, which is characterized by the following interwining relation.
\begin{prop}[Theorem 3.25 of \cite{BGJ}]\label{alg}Let $T_t:\M\to\M$ be a symmetric quantum Markov semigroup and let $(\A,\hat{\M},\delta)$ be a derivation triple of $T_t$. Suppose that there exists a symmetric quantum Markov semigroup $\hat{T}_t:\hat{\M}\to \hat{\M}$ such that for some $\la\in \mathbb{R}$ and any $t\ge 0$,
\begin{align}\label{al} \pl \pl \tilde{T}_t|_\M=T_t\pl , \pl \text{and}\pl \pl \delta\circ T_t=e^{-\la t}\hat{T}_t\circ \delta\pl.\end{align}
 Then $T_t$ satisfies $\ARic\ge \la$  with $\ric=\la \id$ as a multiple of the identity operator.
 \end{prop}
We call the equation \eqref{al} as $\la$-$\ARic$ to specify the relation $\ric=\la \id$.

Another curvature condition motivated from optimal transport is the entropy Ricci curvature lower bound introduced in \cite{CM} (see also \cite{DR,CM18}). Such entropy Ricci curvature lower bound is defined via the $\la$-geodesic convexity of entropy $H$ with respect to Wasserstein distance. Here we recall a related condition of gradient estimate. Let $T_t:\M\to \M$ be a symmetric quantum Markov semigroup and $(\A,\hat{\M},\delta)$ be a derivation triple for $T_t$. For simplicity of notation, we write $\tau$ for the trace on both $\M$ and $\hat{\M}$.
For a state $\rho\in S(\M)$,
we define the weighted $L_2$-(semi)norm on $\hat{\M}$ as
 \[ \norm{x}{\rho}:
 \lel  \left(\int_0^1  \tau(x^* \rho^{1-s}x \rho^s) ds \right)^{1/2}\pl .\]
We denote $L_2(\hat{\M},\rho)$ as the completion of $\hat{\M}$ under this norm. We recall the following definitions from \cite{wirth,WZ}.
\begin{defi}\label{GE}
We say $T_t$ satisfies $\la$-gradient estimate ($\la$-GE) if for any $\rho\in S(\M)$ and $x\in \A_\E$ with $ E(x)=0$,
\[\norm{\delta(T_t(x))}{\rho}\le e^{-\la t} \norm{\delta(x)}{T_t(\rho)}\pl, \forall t\ge 0\pl.\]
We say $T_t$ satisfies $\la$-complete gradient estimate ($\la$-CGE) if for any finite von Neumann algebra $\R$, $\id_\R\ten T_t$ satisfies $\la$-GE.
\end{defi}

In finite dimensional cases,
$\la$-GE is shown to be equivalent to the $\la$-geodesic convexity of $H$, i.e., the entropy Ricci curvature lower bound \cite{CM18,DR}. On finite von Neumann algebras, $\la$-GE is a sufficient condition for $\la$-geodesic convexity of $H$ \cite[Theorem 7.12]{wirth}. It was proved in \cite{JLLR} that $\la$-$\ARic$ is stronger than $\la$-CGE.
 \begin{prop}[Proposition 3.6 of \cite{JLLR}]For $\la\in \mathbb{R}$,
$\la$-$\ARic$ implies $\la$-CGE
\end{prop}

Up to this writing, it is not clear whether $\la$-$\ARic$ and $\la$-CGE are equivalent. One observation suggesting the negation is that
CGE is independent of specific choice derivation triple and only determined by the semigroup $T_t$ \cite[Proposition 3.14]{BGJ}, while such independence is not clear for $\ARic$.

The complete gradient estimates are introduced in \cite{WZ} and several examples in this paper are independently studied there. Here we emphasis the difference between our work and \cite{WZ}. In our discussion we always need a derivation triple into a larger von Neumann algebra $\hat{\M}$, whereas \cite{WZ} uses derivation $\delta:\dom(A^{1/2})\to \H$ into a $\M$-bimodule $\H$, which is based on the representation theorem \cite[Theorem 8.2 \& 8.3]{CS03} by Cipriani and Sauvageot. The more special subalgebra structure enables us to prove geometric curvature lower bound $\ARic$ that is stronger than CGE. Most of examples in our discussion will be given with concrete construction of derivation triple.

On the other hand, $\la$-$\ARic$ implies the $\Gamma_2$-condition studied in \cite{JZ}. Assume that $\A$ is a w$^*$-dense subalgebra invariant under the generator $A$, i.e., $\A\subset \dom(A)$ and $A(\A)\subset\A$. Recall that the $\Gamma_2$ operator of $T_t=e^{-At}$ is given by
\[ \Gamma_2(x,y)=\frac{1}{2}\Big(\Gamma(Ax,y)+\Gamma(x,Ay)-A\Gamma(x,y)\Big)\pl,\pl x,y\in \A\pl.\]
We say $T_t$ satisfies $\Gamma_2\ge \la \Gamma$ for $\la\in \mathbb{R}$ if for any finite sequence $(x_j)_{k=1}^n\subset \A_\E$,
\[[\Gamma_2(x_j,x_k)]_{j,k=1}^n\ge \la  [\Gamma(x_j,x_k)]_{j,k=1}^n \]
as elements in $M_n(\M)$.

\begin{prop}
Let $(\A,\hat{\M},\delta)$ be a derivation triple of $T_t=e^{-At}$. Assume that $\A$ is a w$^*$-dense subalgebra invariant under the generator $A$. Then
$\la$-$\ARic$ implies $\Gamma_2\ge \la \Gamma$.
\end{prop}
\begin{proof}For $x\in \A$ and $z\in \M$ positive,
\begin{align*}
\tau(z\Gamma(x,x))=&\tau(z\delta(x)^*\delta(x))\\
2\tau(z\Gamma_2(x,x))=&\tau(z\delta(Ax)^*\delta(x))+\tau(z\delta(x)^*\delta(Ax))-\tau((Az)\delta(x)^*\delta(x))
\\ =&\tau\left(z(\hat{A}\delta(x)+\ric(\delta(x)))^*\delta(x)\right)+\tau\left(z\delta(x)^*(\hat{A}\delta(x)+\ric(\delta(x)))\right)
\\ &-\tau\left((Az)\delta(x)^*\delta(x)\right)
\\ =&\tau\left(z(\hat{A}\delta(x))^*\delta(x)\right)+\tau\left(z\delta(x)^*(\hat{A}\delta(x))\right)-\tau\left((Az)\delta(x)^*\delta(x)\right)
\\&+\tau\left(z(\ric\delta(x))^*\delta(x)\right)+\tau\left(z\delta(x)^*(\ric\delta(x))\right)
\end{align*}
Note that
\begin{align*}&\tau\left(z(\hat{A}\delta(x))^*\delta(x)\right)+\tau\left(z\delta(x)^*(\hat{A}\delta(x))\right)-\tau\left((Az)\delta(x)^*\delta(x)\right)
\\=&\lim_{t\to 0}\frac{1}{t} \tau\left(z \big(\hat{T_t}(\delta(x)^*\delta(x))-\hat{T_t}(\delta(x))^*\hat{T_t}(\delta(x))\big)\right) \ge 0 \end{align*}
and
\begin{align*}\tau\left(z(\ric\delta(x))^*\delta(x)\right)
=& \lan \ric(\delta(x)z^{1/2}) ,\delta(x)z^{1/2}\ran
\\ \ge& \la\lan \delta(x)z^{1/2} ,\delta(x)z^{1/2}\ran
\\= & \la\tau(z\delta(x)^*\delta(x))=\la\tau(z\Gamma(x,x))
\end{align*}
and similarly for $\tau\left(z\delta(x)^*\left(\ric\delta(x)\right)\right)$. The same argument applies to $\id \ten T_t$, which completes the proof.
\end{proof}
As observed in \cite{WZ}, $\Gamma_2\ge \la \Gamma$ corresponds to the (complete) gradient estimate similar to Definition \ref{GE} for the weighted norm
\[\vertiii{x}_{\rho}=\tau(x^*x\rho)^{1/2}\pl.\]
In the noncommutative case, this is unlikely to equivalent to GE.
\subsection{CB-return time}
We review the main theorem from \cite{BGJ}, which the key ingredient that enables us to obtain CLSI from non-positive curvature lower bound is the CB-return time. Let $\M$ be a finite von Neumann algebra and $\N\subset\M$ be a subalgebra. The conditional $L_\infty$ space $L_\infty^1(\N\subset\M)$ is defined as the completion of $\M$ with respect to the norm
\[ \norm{x}{L_\infty^1(\N\subset\M)}=\sup_{a,b\in L_2(\N)\pl,\|a\|_2=\|b\|_2=1}\norm{axb}{1}\pl,  \]
where the supremum takes over all $a,b\in L_2(\N)$ with $\|a\|_2=\|b\|_2=1$. It is clear that for $\N=\mathbb{C}1$, $L_\infty^1(\N\subset\M)$ is simply $L_1(\M)$.
The operator space structure of $L_\infty^1(\N\subset\M)$ is given by
\[M_n(L_\infty^1(\N\subset\M))=L_\infty^1(M_n(\N)\subset M_n(\M))\pl.\]
(see \cite{JPMemo} and \cite[Appendix]{GJL19}).
\begin{defi}
Let $T_t:\M\to \M$ be a symmetric quantum Markov semigroup and $\N$ be its fixed point subalgebra with conditional expectation $E:\M \to \N$. The complete bounded (CB) return time of semigroup $T_t$ is defined as
\[t_{cb}:=\inf \{ \pl t\ge 0 \pl |\pl \norm{T_t-E:L_\infty^1(\N\subset\M)\to L_\infty(\M)}{cb}\le 1/2\}\]
If such $t$ does not exist, we write  $t_{cb}=+\infty$.
\end{defi}

Define the function
\[\kappa(\la,t)=\begin{cases}
                 \frac{1}{4t}, & \mbox{if } \la=0 \\
                 \frac{\la}{2(1-e^{-2\la t})}, & \mbox{if } \la\neq 0.
               \end{cases}.\]
The following is Corollary 3.28 from \cite{BGJ}.
\begin{theorem}\label{CLSI3}
Let $T_t:\M\to \M$ be a symmetric quantum Markov semigroup. Suppose
\begin{enumerate}
\item[i)] $T_t$ satisfies $\ARic\ge \la$ for some $\lambda \in \mathbb R$;
\item[ii)] $T_t$ has finite CB-return time $t_{cb} < \infty$.
\end{enumerate}
Then $T_t$-satisfies $\kappa(\la,t_{cb})$-CLSI.
\end{theorem}
The condition i) can be weaken to $\la$-CGE as argued in \cite[Corollary 3.20]{BGJ}. In this paper, we will mostly argue through the stronger condition ``GRic''.

\subsection{Compact quantum groups}  We refer to \cite{Wo98,NT13} as standard references for the basic facts on compact quantum groups. We write $\ten_{min}$ for the $C^*$-minimal tensor product and $\overline{\ten}$ for the von Neumann algebra tensor product.
A $C^*$-algebraic {\it compact quantum group (CQG)} $\G$ is a pair $(C(\G),\Delta)$ where $C(\G)$ is a unital $C^*$-algebra and  $\Delta: C(\G)\to C(\G)\ten_{min} C(\G)$ is a unital $*$-homomorphism (called the {\it comultiplication}) which satisfies \begin{itemize}
\item[i)] co-associativity: $(\D\ten id)\Delta=(id\ten \Delta)\Delta$
\item[ii)] cancellation property: $\Delta(C(\G))(1\ten C(\G))$ and $\Delta(C(\G))( C(\G)\ten 1)$ are total in $C(\G)\ten_{min} C(\G)$.
\end{itemize}
There exists a unique Haar state $h:C(\G)\to \mathbb{C}$ such that
\[(id\ten h)\Delta(a)=h(a)1=(h\ten id)\Delta(a)\pl, \pl \forall a\in C(\G).\]
Let $\la:C(\G)\to B(L_2(\G))$ be left regular representation on the GNS Hilbert space $L_2(\G)=L_2(C(\G),h)$. We denote $
\lambda(C(\G))$ by $C_r(\G)$ and called it the {\it reduced C$^\ast$-algebra} of continuous functions on $\G$.  We also denote by $L_\infty(\G)$ the von Neumann algebra generated by $C_r(\G)$ in $B(L_2(\G))$. Then $\Delta$ extends normally to $L_\infty(\G)$ and $(L_\infty(\G),\Delta,h)$ is a von Neumann algebraic compact quantum group.

A {\it (finite-dimensional) representation} of a CQG $\G$ is an invertible element $u \in B(H) \otimes C(\G)$ (where $H$ is a finite-dimensional Hilbert space) which satisfies
\[
(\id \otimes \Delta)u = u_{12}u_{13} \in B(H) \otimes C(\G) \otimes C(\G).
\]
Here we use the standard leg numbering notation.  Note that if we choose an a basis $(e_i)_{1 \le i \le d}$ for $H$ and write $u = [u_{ij}] \in M_d(C(\G))$ relative to this basis, then the above formula simply says that
\[
\Delta(u_{ij}) = \sum_k u_{ik} \otimes u_{kj}, \pl (1 \le i,j \le d).
\]
We say that $u$ is {\it unitary} if $uu^* = u^*u  = 1$.
Given two representations of $\G$, say $u \in B(H_u) \otimes C(\G)$ and $v \in B (H_v) \otimes C(\G)$, we can define their {\it direct sum} $u \oplus v \in B(H_u \oplus H_v) \otimes C(\G)$ in the obvious way and their {\it tensor product} as $u \otimes v = u_{13}v_{23} \in B(H_u \otimes H_v) \otimes C(\G)$.  We denote by $\text{Mor}(u,v) = \{T \in B(H_u,H_v): \pl (T \otimes \id)u = v(T \otimes \id)\}$.  Such a $T$ is said to {\it intertwine} $u$ and $v$.  We say $u$ and $v$ are equivalent if $\text{Mor}(u,v)$ contains an invertible element.  We say that $u$ is {\it irreducible} if $\text{Mor}(u,u) = \mathbb C \id.$  We also note that every representation $u$ is equivalent to a unitary representation, and every unitary representation is a direct sum of irreducible representations.

We denote by $\mathcal O(\G) \subset C(\G)$ the collection of all matrix elements of finite dimensional representations of $\G$.  I.e., $x \in \mathcal O(\G)$ is and only if $x = (\omega_{\xi,\eta}\otimes \id)u$ for some representation $u \in B(H_u) \otimes C(\G)$, $\xi,\eta \in H_u$.  Then $\mathcal O(\G)$ is a dense unital $\ast$-subalgebra of $C(\G)$ on which the Haar state is faitful, the comultiplication restricts to a morphism $\Delta:\mathcal O(\G) \to \mathcal O(\G) \otimes \mathcal O(\G)$ (algebraic tensor product), and in fact $(\mathcal O(\G),\Delta)$ becomes a Hopf $\ast$-algebra with counit $\epsilon:\mathcal O(\G) \to \mathbb C$ and coinverse $S: \mathcal O(\G) \to \mathcal O(\G)^{op}$ given by
\[
(\epsilon \otimes \id)u = \id, \ (S \otimes \id )u = u^{-1} \quad \text{for any representation $u \in  B(H_u) \otimes C(\G)$}.
\]  Denote by $\text{Irr}(\G)$ the set of irreducible unitary representations of $\G$ up to unitary equivalence.
Choosing representatives $(u^\pi)_{\pi \in \text{Irr}(\G)}$ and orthonormal bases $(e^\pi_i)_{1 \le i \le \dim (\pi)} \subset H_\pi$, it follows that
\[
\{u^\pi_{ij}: \pl 1 \le i,j \le \dim \pi\}_{\pi \in \text{Irr}(\G)}
\]
is a linear basis for $\mathcal O(\G)$.

In this paper we will focus on compact quantum group of {\it Kac-type}, that is, CQGs $\G$ for which the Haar state $h$ is a trace. In this case we typically write $h=\tau$.  In this special situation, the above basis for $\mathcal O(\G)$ is an orthogonal basis for $L^2(\G)$.  More precisely we have
\[
\tau((u^\pi_{ij})^*u^\sigma_{kl}) = \frac{\delta_{\pi,\sigma}\delta_{ik}\delta_{jl}}{\dim\pi}.
\]
Moreover, when $\G$ is of Kac type, the antipode $S$ extends to a normal $\ast$-isomorphism $S:L_\infty(\G) \to L_\infty(\G)^{op}$.

\section{Central semigroups on compact quantum groups}
Let $\G$ be a compact quantum group of Kac type. Whenever we speak of quantum Markov semigroups on $L_\infty(\G)$, we mean Markovian with respect to the canonical Haar trace $\tau$ on $L_\infty(\G)$.
\begin{defi}
We say a quantum Markov semigroup $T_t:L_\infty(\G)\to L_\infty(\G)$ is called {\it central} if for all $t \ge 0$, $T_t$ is both left and right invariant, i.e.
\[ \Delta \circ T_t= (T_t\ten \id) \circ \Delta =(\id\ten T_t) \circ \Delta \pl.\]
\end{defi}
Following the group case in \cite{BGJ}, we show that central semigroups satisfy $\ARic \ge 0$.
\begin{theorem}\label{central}Let $\G$ be a compact quantum group and let $T_t=e^{-At}:L_\infty(\G)\to L_\infty(\G)$ be a symmetric quantum Markov semigroup. If $T_t$ is central, then $T_t$ satisfies $\ARic \ge 0$. The CGE$\ge 0$ for central semigroups is independently obtained in \cite[Example 3.12]{WZ}.
\end{theorem}
\begin{proof}The central property $\Delta\circ T_t =(\id_\G\ten T_t)\circ \Delta=(T_t\ten \id_\G)\circ \Delta$ translates to the following commutative diagram.
\begin{equation}\label{ccss}
 \begin{array}{ccc}  L_\infty(\G)\overline{\ten} L_\infty(\G)\pl\pl &\overset{ \operatorname{id}_{\G}\ten T_t\text{\pl or \pl} T_t\ten \id_\G }{\longrightarrow} & L_\infty(\G)\overline{\ten} L_\infty(\G) \\
                    \uparrow \Delta    & & \uparrow \Delta \\
                     L_\infty(\G)\pl\pl &\overset{T_t}{\longrightarrow} & L_\infty(\G)
                     \end{array} \pl .
                     \end{equation}
                      Let $(\A,\M,\delta)$ be a derivation triple for $T_t$ such that
\[E_\G(\delta(x)^*\delta(y))=\Gamma(x,y)\pl.\]
where $E_\G:\M \to L_\infty(\G)$ is the conditional expectation from $M$ to $L_\infty(\G)$.
 We show that $\nabla=(\delta\ten \id)\circ \D:L_{\infty}(\G)\to \M\bar\ten L_{\infty}(\G)$ is also a derivation for $T_t=e^{-At}$. First, for the generator $A$ we have $E_\Delta(A\ten \id_\G)\Delta=A$ where $E_\Delta: L_\infty(G)\overline{\ten}L_\infty(G)\to L_\infty(G)$ is the conditional expectation obtained as the adjoint of $\Delta$. Then, for $x,y\in \A$
 \begin{align*} \Gamma_A(x,y)&=x^*Ay+(Ax)^*y-A(x^*y)
 \\&=x^*E_\Delta(A\ten \id)\Delta(y)+(E_\Delta(A\ten \id)\Delta(x))^*y-E_\Delta(A\ten \id)\Delta(x^*y)
 \\&=E_\Delta(\Delta(x)^*A\ten \id)\Delta(y)+(A\ten \id)\Delta(x)^*\Delta(y)-(A\ten \id)\Delta(x^*y))\\ &=E_\Delta(\Gamma_{A \otimes \id}(\Delta(x),\Delta(y)))=E_\Delta\circ (E_\G \otimes \id)((\delta\ten \id)\Delta(x),(\delta\ten\id)\Delta(y))
\end{align*}
where we have used the fact
 $(\A\ten L_\infty(\G), \M\overline \ten {L_\infty(\G)}, \delta\ten \id)$ is a derivation triple for $T_t\ten \id_\G$. Now for the new derivation $\nabla=(\delta\ten \id_\G)\circ \Delta$, we have $0$-$\ARic$ relation
 \[ \nabla\circ T_t=(\delta\ten \id_\G)\circ\Delta\circ T_t=(\delta\ten \id_\G)(\id_\G\ten T_t)\circ \Delta=(\id_\M\ten T_t)(\delta\ten \id_\G)\circ \Delta=(\id_\M\ten T_t)\nabla\pl.\]
 where $\id_\M\ten T_t$ is the extension semigroup of $T_t$ on $\M\overline{\ten} L_{\infty}(\G)$.
\end{proof}
\subsection{Fourier multipliers on group von Neumann algebras}\label{f}
In this subsection, we consider group von Neumann algebras as particular examples of co-commutative compact quantum groups. Let $G$ be a discrete group. The left regular representation of $G$ is given by \[\la: G\to B(l_2(G)), \la(g)\ket{h}=\ket{gh}\] where $\{\ket{h}| h\in G\}$ is the standard orthonormal basis of $l_2(G)$. The group von Neumann algebra \[\L(G)=\overline{\text{span} \{\la(g)\pl | \pl g\in G \}^{w^*}}\subset B(l_2(G))\] is a finite von Neumann algebra equipped with the canonical trace $\tau(\sum_{g}a_g\la(g))=a_e$ where $e$ is the identity element of $G$. $\L(G)$ has the structure of a von Neumann-algebraic compact quantum group (of Kac type) when equipped with the comultiplication map given by
\[ \D:\L(G)\to \L(G)\overline{\ten}\L(G)\pl, \D(\la(g))=\la(g)\ten \la(g) \pl, \pl \forall \pl g\in G.\]
Here  the underlying C$^\ast$-algebraic CQG is $\G = \hat G= (C^*(G), \Delta)$ with Haar trace $h = \tau$.  Then $C_r(\hat G) = C^*_r(G)$, $L_\infty(\hat G) =\L(\G)$, and $\mathcal O(\hat G) \cong \mathbb C G$, the group algebra of $G$. We say that $\hat G$ is the compact quantum group {\it dual to} $G$.  This generalizes the compact-discrete Pontryangin duality for abelian groups.

For a function $\phi:G\to \mathbb{C}$, we associate the Fourier multiplier map \[T_\phi: \L(G)\to \L(G)\pl, \pl T_t(\la(g))=\phi(g)\la(g)\pl.\]
In general, $T_\phi$ is of course only defined on the $\sigma$-weakly dense $\ast$-subalgebra $\lambda(\mathbb C G)$.
Note that all Fourier multipliers are central in the sense of the previous section: \[ (T_\phi\ten \id)\circ \D(\la(g))=(\id \ten T_\phi)\circ \D(\la(g))=\phi(g)\la(g)\ten \la(g)\pl.\]
Conversely, it is also clear that all central map $T:\lambda(\mathbb C G) \subseteq \L(G)\to \L(G)$ have to be Fourier multipliers.

Now let $T_t:\L(G)\to \L(G)$ be a semigroup of (bounded) Fourier multipliers.  Then one can write
\[T_t(\la(g))=e^{-t\psi(g)}\la(g), \pl g \in G,\]
where the generator $A$ of the semigroup is given by $A(\la(g))=\psi(g)\la(g)$,  i.e., a (generally unbounded) multiplier associated to some function $\psi:G\to \mathbb{C}$. Recall that $\psi$ is called  {\it conditionally negative definite} if for any finite sequence $\sum_{i=1}^nc_i=0, c_i\in \mathbb{C}$ and $g_1,\cdots,g_n\in G$,
    \[ \sum_{i,j=1}^nc_i\bar c_j\psi(g_i^{-1}g_j)\le 0\]
It is known that $T_t$ is a symmetric quantum Markov semigroup if and only if $\psi$ is a real-valued conditionally negative definite function with $\psi(e)=0$ and $\psi(g)=\psi(g^{-1})$.

It then follows from Theorem \ref{central} that any symmetric Markov semigroup of Fourier multipliers satisfies $\Ric\ge 0$.
In the following, we give a explicit construction of a derivation triple for $T_t$. The idea is inspired from the Markov dilation of Fourier multiplier semigroups from \cite{fm}. Recall that by Schoenberg Theorem \cite{schoenberg38}, there exist a real Hilbert space $\H$ and an affine isometric action $\beta : G \to \text{Isom}(\mathcal{H}), g \mapsto  \beta_g$ such that
\[\psi(g)= \norm{\beta_g({\bf 0})}{\mathcal{H}}^2\pl ,\pl g\in G\pl.\]
Here ${\bf 0}$ is the zero vector. For any $v\in \H$, one can write
\[\beta_g(v)=\pi_g(v)+b(g)\]
where $\pi:G\to O(\mathcal{H})\pl, g\mapsto \pi_g$ is an orthogonal representation of $G$ and $b:G\to \H$ is a $1$-cocycle with respect to $\pi$,
\[b(gh)=b(g)+\pi_g(b(h)) \pl g,h \in G.\]
Then $\psi(g)=\norm{b(g)}{\mathcal{H}}^2$. Then the gradient form can be expressed as
\begin{align*}
\Gamma(\la(g),\la(h))&=\frac{1}{2}\Big((A\la(g))^*\la(h)+\la(g)^*(A\la(h))-A(\la(g^{-1}h))\Big)
\\&=\frac{1}{2}(\norm{b(g)}{}^2+\norm{b(h)}{}^2-\norm{b(g^{-1}h)}{}^2)\la(g^{-1}h)
\\&=\frac{1}{2}(\norm{b(g^{-1})}{}^2+\norm{\pi_{g^{-1}}(b(h))}{}^2-\norm{b(g^{-1})+\pi_{g^{-1}}(b(h))}{}^2)\la(g^{-1}h)
\\&=-\lan b(g^{-1}),\pi_{g^{-1}}(b(h))\ran\la(g^{-1}h)=\lan b(g),b(h)\ran\la(g^{-1}h),
\end{align*}
where we used that $b(g)+\pi_{g}(b(g^{-1}))=b(e)=0$. For a real Hilbert space $\mathcal{H}$, an $\mathcal{H}$-isonormal process on a standard probability
space $(\Omega, m)$ is a linear mapping $W: \mathcal{H} \to L_0(\Omega)$ satisfying the following properties:
\begin{enumerate}
\item[i)] for any $v\in \mathcal{H}$, the random variable $W(v)$ is a centered real Gaussian.
\item[ii)] for any $v_1,v_2\in \mathcal{H}$, we have $E_\Omega(W(v_1)W(v_2))=\lan v_1,v_2\ran_\mathcal{H} $
\item[iii)] The linear span of the products $\{W(v_1)W(v_2)\cdots W(v_n)\pl |\pl v_1,\cdots, v_n\in \mathcal{H}\}$ is dense in the real Hilbert space $L_2(\Omega)$
\end{enumerate}
Here $L_0(\Omega)$ is the space of measurable functions on $\Omega$. %Indeed, let $(e_i)_{i\in I}$ be a ONB of $\H$ and $(\gamma_i)_{i\in I}$ be an i.i.d family of standard Gaussian. One can defined that for $v\in \H$
%\[ W(v)=\sum_{i}\lan e_i,v \ran\gamma_i\pl.\]
From the properties of the Gaussian distribution,
\[E_\Omega(e^{-itW(v)})=e^{-\frac{t}{2}\|v\|_{\mathcal{H}}^2}\pl, t\in \mathbb R, v\in \mathcal{H}\pl.\]
Given an orthogonal transformation $T:\mathcal{H}\to \mathcal{H}$, the quantization map \begin{align*}\Gamma(T): L_\infty(\Omega)\to L_\infty(\Omega)\pl, \pl \Gamma(T)(e^{iW(v)})=e^{iW(Tv)}\pl, \\\pl \Gamma(T)(W(v_1)\cdots W(v_n))=W(Tv_1)\cdots W(Tv_n)\pl.\end{align*}
is a normal $*$-automorphism. Then the orthogonal representation $\pi:G\to O(\H)$ induces an action $\al$ of $G$ on $L_\infty(\Omega)$ as follows
\[\al_s(e^{iW(v)})=\Gamma(\pi_s)(e^{iW( v)})=e^{iW( \pi_sv)}\pl, \pl s\in G\]
Let $\M=L_\infty(\Omega)\rtimes_\al G$ be the crossed product given by the action $\al$. $\M$ is again a finite von Neumann algebra equipped with the extension trace
\[\tau_\M(a\rtimes\la(g))=\begin{cases}
                            E_\Omega(a), & \mbox{if } g=e \\
                            0, & \mbox{otherwise}.
                          \end{cases}\pl.\]
Denote $\mathbb{C}G=\text{span}\{\la(g)\}\subset \L(G)$ as the group algebra. We define the following derivation
$\delta: \mathbb C G\to L_2(\M)\cong (L_2(\Omega)\otimes_2 L_2(\L(G)))$ by
\[\delta(\la(g))=iW(b(g))\rtimes \la(g)\pl.\]
This is a derivation because
\begin{align*}\delta(\la(g))\la(h)+\la(g)\delta(\la(h))&=(W(b(g))\rtimes \la(g))\la(h)+\la(g)(W(b(h))\rtimes \la(h))\\
&=W(b(g))\rtimes \la(gh)+W(\pi_{g}b(h))\rtimes \la(gh)
\\
&=(W(b(g))+W(\pi_{g}b(h)))\rtimes \la(gh)=W(gh)\rtimes \la(gh).
\end{align*}
Moreover,
\begin{align*}
E_G(\delta(\la(g))^*\delta(\la(h)))&=E(\la(g)^*W(b(g))W(b(h)) \la(h))\\&=\la(g)^*E(W(b(g))W(b(h)))\la(h)\\&=\lan b(g),b(h)\ran\la(g^{-1}h)=\Gamma(\la(g),\la(h))
\end{align*}
Thus we have shown that $(\mathbb{C}G,\M,\delta)$ is derivation triple. It was proved in \cite[Proposition 4.1]{MTT} that $T_t:\L(G)\to \L(G) $ admits an extension $\hat{T}_t: L_\infty(\Omega)\rtimes_\al G\to L_\infty(\Omega)\rtimes_\al G$,
\[ \hat{T}_t(a\rtimes \la(g))=e^{-\psi(g)t}a\rtimes \la(g)\pl.\]
$\hat{T}_t$ is complete positive because $\hat{T}_t$ is unital and $\norm{\hat{T}_t}{cb}=\norm{T_t}{cb}=1$. It is clear that
$\hat{T}_t$ forms a symmetric Markov semigroup satisfying the algebraic relation
\[ \hat{T}_t\circ \delta=\delta\circ T_t \pl.\]
This verifies that $T_t$ has $\ARic \ge 0$ (actually $0$-$\ARic$). To ensure the CB-return time is finite, we need some growth condition in $\psi$.

\begin{theorem}\label{fourier}Let $G$ be a discrete group and $T_t:\L(G)\to \L(G)$ be a symmetric quantum Markov semigroup of Fourier multipliers
\[T_t: \L(G)\to \L(G)\pl, \pl T_t(\la(g))=e^{-t\psi(g)}\la(g)\]
given by a conditionally negative definite function $\psi:G\to \mathbb{R}$. Then
$T_t$ satisfies $\Ric \ge 0$. The fixed-point subalgebra is $\N = \la(H)'' \cong \L(H)$
where $H$ is the subgroup $\{g\in G\pl |\pl \psi(g)=0\}$. If in additional, $\psi$ satisfies
 \begin{enumerate}
\item[i)]the growth condition:
$\text{for some}\pl r>0\pl, \pl C_r=\sum_{g\notin H} r^{\psi(g)}<+\infty\pl,$
\item[ii)] the spectral gap condition:
$\si=\inf_{g\notin H} \psi(g)>0$.
\end{enumerate}
 Then $T_t$ satisfies $\la$-CLSI for
    \[\la= \Big(4 \si^{-1}\log (2C_r)-4\log r \Big)^{-1}\]
\end{theorem}
\begin{proof}
It suffices to prove the following estimate that for $t> -\log r$,
   \begin{align}\label{1}\norm{T_t-E_\N:L_1(\L(G))\to L_\infty(\L(G))}{cb}\le e^{-\si(t+\log r)}C_r\end{align}
where $E_\N$ is the conditional expectation onto $\N$. Note that $\{\la(g)\pl|\pl g\in G\}$ is an ONB of $L_2(\L(G))$. Then the Choi operator of $T_t$ and $E_\N$ is
\begin{align*}
&C(T_t)=\sum_{g\in G}e^{-\psi(g)t}\la(g)^{op}\ten \la(g)\in \L(G)^{op}\overline{\ten} \L(G)
\\&C(E_\N)=\sum_{g\in H}\la(g)^{op}\ten \la(g)\in \L(G)^{op}\overline{\ten} \L(G)
\end{align*}
Then by Effros-Ruan Theorem \cite{ER90},
\begin{align*}
\norm{T_t-E_\tau:L_1(\L(G))\to  L_\infty(\L(G))}{cb}
= &\norm{C(T_t)-C(E_\tau)}{\L(G)^{op}\overline{\ten} \L(G)}
\\= &\norm{\sum_{g\notin H}e^{-\psi(g)t}\la(g^{-1})^{op}\ten \la(g)}{\L(G)^{op}\overline{\ten} \L(G)}
\le \sum_{g\notin H}e^{-\psi(g)t}
\end{align*}
provided the sum is finite. By the growth condition and spectral gap condition,
we have
\begin{align*}
\norm{T_t-E_\N:L_1(\L(G))\to  \L(G)}{cb}
\le &\sum_{g\notin H}e^{-\psi(g)t}= \sum_{g\notin H}e^{\psi(g)\log r} e^{-\psi(g)(t+\log r)}
\\ =& e^{-\si(t+\log r)}C_r
\end{align*}
is finite for $t>-\log (r)$.
Therefore we have the cb-return estimate
\[t_{cb} \le \si^{-1}\log (2C_r)-\log r \]
The CLSI constant follows from Theorem \ref{CLSI3}.
\end{proof}
\begin{rem}{\rm For ergodic semigroups, this above theorem is comparable to \cite[Theorem B]{J+HC}. On one hand, we know that hyperconctractivity, or equivalently LSI, implies MLSI.
On the other hand, our growth condition here is weaker to \cite[Theorem B]{J+HC}, and moreover, the above Theorem \ref{fourier} implies MLSI for $T_t\ten \id_\R$ for any finite von Neumann algebra $\R$.}\end{rem}

Following \cite{J+HC}, we have the following variant of the assumptions on the growth of $\psi$.
\begin{cor}Let $G$ be a countable discrete group and let $\psi:G\to \mathbb{R}$ be a real conditionally negative definite function. Suppose $\psi$ satisfies the following two conditions. \begin{enumerate}
\item[i)]Exponential growth condition: there exists a $R>0$ such that for any $s>0$
\[\ \pl |\{g\notin H | \psi(g)\le s+1\}|\le CR^s \pl,\]
\item[ii)] Spectral gap condition:
\[\si=\inf_{g\notin H} \psi(g)>0\]
\end{enumerate}
Then $T_t$ satisfies $\la$-CLSI for $\la=\si\Big(4\log \big(2C+2C(\frac{2}{R})^{\si}\big)\Big)^{-1}$.
\end{cor}
\begin{proof}
It sufficient to note that for $0<r<R^{-1}<1$,
\begin{align*}
&\sum_{g\notin H}r^{\psi(g)}\le\sum_{\si \le \psi(g)\le 1}r^{\si}+\sum_{m=1}^\infty \sum_{g\notin H, m\le \psi(g) \le (m+1)}r^m
 \le  Cr^{\si}+\sum_{m=1}^\infty CR^m r^{m}
 \le Cr^{\si}+C\frac{ Rr}{1-Rr}\pl.
\end{align*}
Then by Theorem \ref{fourier}, $T_t$ satisfies $\la$-CLSI with
\[ \la=\si\Big(4\log (2Cr^{\si}+2C\frac{ Rr}{1-Rr})-\log r\Big)^{-1}=\si\Big(4\log (2C+2C\frac{ Rr}{(1-Rr)r^{-\si}}\Big)^{-1}\]
Choose $r=R/2$, we obtain $\la=\si\Big(4\log \left(2C+2C(\frac{2}{R})^{\si}\right)\Big)^{-1}$.
\end{proof}
\begin{exam}{\rm Let $G$ be a finitely generated group and $|\cdot|:G\to \mathbb{Z}_+$ be the word length function to relative to some fixed finite symmetric generating set $S$. Suppose $|\cdot|$ is conditionally negative definite. We have
\begin{enumerate}
\item[i)] the fixed point subgroup $H=\{e\}$ is trivial.
\item[ii)] $G$ is at most exponential growth with respect to $|\cdot|$: $\{g| |g|\le n+1\}|\le |S|(|S|-1)^n$.
    \item[iii)] the spectral gap $\si=1$.
\end{enumerate}
Then the word length semigroup \[P_t:\L(G)\to \L(G)\pl, \la(g)\mapsto e^{-|g|t}
\la(g)\]
satisfies $\la$-CLSI with $\la=\Big(4\log \big(2|S|(\frac{|S|+1}{|S|-1})\big)\Big)^{-1}$.}
\end{exam}
\begin{rem}{\rm The above estimates can be potentially improved if $G$ has the property of rapid decay (RD) with respect to $|\cdot|$.  That is, there exists a polynomial $P \in \mathbb \mathbb{R}_+[x]$ such that for any $d\ge 0$,
\[
\|\lambda(f)\|_{\L(G)} \le P(k)\|f\|_2 \qquad  \forall \pl f\in \mathbb C G, \pl \text{supp}f \subseteq W_d,
\] where $W_d = \{g \in G: \pl |g| = d\}$ are the words of length $d$. Then instead of using triangle inequality for all $g\in G$, one can have
\begin{align*}
\norm{P_t-E_\tau:L_1(\L(G))\to  \L(G)}{cb}
\le &\norm{\sum_{g\neq e}e^{-|g|t}\la(g)^{op}\ten \la(g)}{} \\ \le &\sum_{d\ge 1}e^{-dt}\norm{\sum_{|g|=d }\la(g)}{}\\ =& \sum_{d \ge 1} e^{-dt}P(d)|W_d|^{1/2}.
\end{align*}
Here in the second inequality we have used that the map $\la(g)\mapsto \la(g^{-1})^{op}$ is a $*$-isomorphism and the comultiplication $\Delta(\la(g))=\la(g)\ten \la(g)$ is an injective $*$-homomorphism.
Then $T_t = e^{-|\cdot |t}$ has finite CB-return time whenever $|W_d|$ is at most exponential growth.
}
\end{rem}

%\subsection{Finite group}
\subsection{Free orthogonal quantum groups}
In this section, we consider the {\it free orthogonal quantum groups} $O_N^+$.  The free orthogonal quantum groups were introduced by Wang and Van Daele \cite{Wa95, WVD}, and their representation theory was later studied in detail by Banica \cite{Ba96, Ba97}. % and central quantum Markov semigroup are classfied in \cite{CFK} under name of ad-invariant semigroup.
Fix $N\ge 2$. Let $O_N^+ = (C(O_N^+),\Delta)$ be the compact quantum group where $C(O_N^+)$ be the universal $C^*$-algebra generated by $\{u_{ij}\}$ such that $u_{ij}=u^*_{ij}$ and $u=\sum_{ij}e_{ij}\ten u_{ij}$ is unitary. The comultiplication $\Delta:C(O_N^+)\to C(O_N^+)\ten_{min}C(O_N^+)$ is given by
\[\Delta(u_{ij})=\sum_{k=1}^N u_{ik}\ten u_{kj}\pl.\]
We note that $u$ as defined above is a unitary representation of $O_N^+$, called the {\it fundamental representation}.  The quantum groups $O_N^+$ should be thought of as free analogues of the classical real orthogonal groups $O_N$.  Indeed, for each $N$ we have a surjective $\ast$-homomorphism $C(O_N^+) \to C(O_N)$ given by $u_{ij}\mapsto v_{ij}$, where $v_{ij} \in C(O_N)$ is the standard coordinate function on $O_N$.  This morphism respects the (quantum) group structures and naturally realizes $O_N$ as a quantum subgroup of $O_N^+$.

We now briefly summarize the representation theory of $O_N^+$, ($N \ge 2$) \cite{Ba96}:  It is known that $\Ir(O_N^+)$ can be indexed by non-negaitve integers $k\in \mathbb{N}_0$ in such a way that and the fusion rules are given by
\[k\ten m \cong |k-m|\oplus (|k-m|+2)\oplus \cdots \oplus (k+m)\pl, k,m \in \mathbb{N}_0 \pl.\]
The labeling $k \mapsto u^k$ can be chose so that $u^{0} = 1$ (the trivial representation) and $u^1 = u \in B(\mathbb C^N) \otimes C(O_N^+)$ is the fundamental representation.  Denote by $n_k = \dim k = \dim u^k$ the dimension of $k$-th irreducible representation, we then have from the fusion rules recursive relation
\[n_kN=n_{k+1}+n_{k-1}\pl.\]
Recall that the (dilated) Chebyshev polynomials ($U_k)_{k \in \N_0} \subset \mathbb R[x]$ of the second kind are defined by the recursion
\begin{align}  \label{chebychev} U_0(x)=1\pl, U_1(x)=x, U_1U_k=U_{k-1}+U_{k+1}\pl.
\end{align}
Comparing the above two recursion relations, one obtains
\[
n_k:=\dim (k) = U_k(N), \pl (k \in \mathbb N_0)
\]
Let $H_k$ be the Hilbert space associated to $u^k$.  We fix an orthonormal basis $(e^k_i)_{1 \le i \le \dim k}$ for $H_k$ and hence get a corresponding matrix representation $u^k = [u^k_{ij}] \in M_{\dim k}(C(O_N^+))$ for each $k$.
%Let $\omega_{ij}^k\in B(H_k)$ be a family of matrix unit
%and $u_{ij}=\omega_{ij}^k\ten \id(U^k)$ be the coefficient of $k$-th irreducible representation. We denote $\mathbb{C}(O_N^+)=\text{span}\{\la(u_{ij}^k)| \pl 1\le i,j\le n_k\pl \}$ as the group algebra.
Then we have that the Hopf-algebra $\mathcal O(O_N^+)$ is spanned by the basis $\{u_{ij}^k\}$, and this basis is orthogonal with respect to Haar trace $\tau$ (which is faithful on $\mathcal O(O_N^+)$):
\[\tau((u_{i'j'}^{k'})^*u_{ij}^k)=\frac{1}{\dim(k)}\delta_{ii'}\delta_{jj'}\delta_{kk'}.\]

In what follows we will slightly abuse notation and identify $\mathcal O(O_N^+) \subseteq L_\infty(O_N^+)$ and $\mathcal O(O_N^+) \subseteq L_2(O_N^+)$ via the usual GNS maps.

\subsubsection{The heat semigroup on $O_N^+$}

In \cite{CFK14}, the symmetric central quantum Markov semigroups on $T_t = e^{At}:L_\infty(O_N^+) \to L_\infty(O_N^+)$ were characterized in terms of their generators \[A: \mathcal O(O_N^+) \to \mathcal O(O_N^+); \pl Au_{ij}^k = \lambda_j(A)u_{ij}^k \qquad \pl, \pl \lambda_k(A) \in \mathbb C.\]     We recall the following theorem from \cite{CFK14} on central semigroups..

\begin{theorem}{\cite[Corollary 10.3]{CFK14}} \label{orth-central}
Let $(\lambda_k)_{k \in \mathbb N_0} \subset \mathbb C$ and define a central semigroup on $\mathcal O(O_N^+)$ via the formula \[T_t:\mathcal O(O_N^+) \to \mathcal O(O_N^+); \pl T_t(u_{ij}^k) = e^{-\lambda_k t}u_{ij}^k.\]  Then $T_t$ extends uniquely to a  symmetric central quantum Markov semigroup $T_t:L_\infty(O_N^+) \to L_\infty(O_N^+)$ if and only if there is a constant $b \ge 0$ and a finite positive Borel measure $\nu$ supported on $[-N,N]$ satisfying $\nu\{N\} = 0$, so that
\begin{align}\label{Hunt-formula}
\lambda_k = \frac{1}{U_k(N)}\Big(bU'_k(N) +\int_{-N}^N \frac{U_k(x)-U_k(N)}{x-N}d\nu(x)\Big).
\end{align}
\end{theorem}
As explained in \cite{CFK14}, the above formula can regarded as a quantum analogue of Hunt's formula for generating functionals of central (=conjugation invariant) L\'evy processes on compact Lie groups.  The measure $\nu$ in the above theorem plays the role of the L\'evy measure in such processes and accounts for the drift term in the L\'evy process.  As in the classical case of compact (connected) Lie groups, we obtain the Laplace-Beltrami operator (or Casimir operator) by setting $\nu = 0$ and choosing a normalization $b = 1$. This led the authors in \cite{CFK14, FHLUZ17} to define the {\it heat semigroup on $O_N^+$} to be central quantum Markov semigroup  $T_t=e^{-At}:L_\infty(O_N^+)\to L_\infty(O_N^+)$ defined by
\[T_t(u_{ij}^k)=e^{-\la_k t}u_{ij}^k\pl, \pl A(u_{ij}^k)=\la_k u_{ij}^k\pl, \pl \la_k=\frac{U_k'(N)}{U_k(N)}.\]
It is well-known that \cite{FHLUZ17} $T_t$ is ergodic.  The heat semigroup on $O_N^+$ has been the subject of intensive study in recent years \cite{CFK14,FHLUZ17, FKS16, Caspers18, BVY}, where hypercontracivity properties and connections to deformation/rigidity for von Neumann algebras were considered.

We note that it is an immediate consequence of Theorem \ref{central} that the heat semigroup on $O_N^+$ satisfies $\Ric \ge 0$.   In the following subsection, we construct a concrete derivation of $T_t$ by uncovering an unexpected connection to the connection between the heat semigroup on $O_N^+$ and the Laplace-Beltrami operator on the classical orthogonal group $O_N$.

\subsubsection{An explicit derivation triple for the heat semigroup on $O_N^+$: Connections to the Laplace operator on $O_N$}

Let $O_N$ be the $N \times N$ real orthogonal group in $M_N(\mathbb R)$. Its Lie algebra $\mathfrak{so}_N$ consists of skew-symmetric matrices and the Lie bracket is given by the commutator $[A,B]=AB-BA$.
Denote by $L=\sum_{j}X_j^2$ as the Casimir operator on $O_N$, where $X_j$ is an orthonormal basis for the negative Killing form $-K(a,b)=(N-2)\text{Tr}(a^tb)$. For the ease of notation, we write $\M=L_\infty(O_N^+)$ in the following.  We denote by $E_\Delta: \M\overline{\ten}\M\to \M$ the conditional expectation given by the adjoint of the comultiplication $\Delta$. Define a right action $\al: O_N\to \text{Aut}(\M)$ by the formula
\[\al_g(u_{ij})=\sum_{1\le l\le N}u_{il}g_{jl}.\]
A priori, $\alpha$ is well-defined as a right action $O_N \to \text{Aut}(C(O_N^+))$, but one readily sees that this action is $\tau$-preserving, and hence extends to a right action $\alpha: O_N \to \text{Aut}(\M)$.  We note that for each $k \in \mathbb N_0$, we have \[\alpha_g(u_{ij}^k)= \sum_{l=1}^{n_k} u^k_{il}(g^{-1})^{k}_{lj} = \sum_{l=1}^{n_k} u^k_{il}g^k_{jl}, \]
where $g \mapsto [g^k_{ij}] \in \mathcal U(H_k)$ is the corresponding representation of $O_N$ on $H_k$ obtained by restricting the representation $u^k$ of $O_N^+$ to the subgroup $O_N$.  In other words, if $\rho: \mathcal O(O_N^+) \to \mathcal O(O_N)$ is the canonical quotient map, then $g^k_{ij} := \rho(u^k_{ij})(g)$.

 We also define the $*$-monomorphism
\[\pi:\M\overline{\ten}\M\to L_\infty(O_N,\M\overline{\ten}\M)\pl, \pi(x)(g)=(\al_{g^{-1}}\ten \id)(x)\pl.\]
and its adjoint conditional expectation \[E_\pi:L_\infty(O_N,\M\overline{\ten}\M)\to \M\overline{\ten}\M\pl, \pl E_\pi(f)=\int_{O_N}(\al_{g}\ten \id) (f(g))dg\pl.\]

\begin{prop}\label{factor}
Let $L=\sum_j X_j^2$ be the Casimir operator on $O_N$ and $\nabla$ be the gradient of $L$ given by
\[\nabla:C^\infty(O_N)\to \bigoplus_{j=1}^{N(N-1)/2} C^\infty(O_N)\pl, \pl \nabla(f)=\bigoplus_{j=1}^{N(N-1)/2} X_jf\pl.\]
Then
\begin{enumerate}
\item[i)] On $\mathcal O(O_N^+)$, we have the identity
\[ E_\Delta\circ E_\pi\circ (L\ten \id_{\M\overline{\ten}\M})\circ  \pi\circ \Delta=\frac{N(N-1)}{2(N-2)} A\pl.\]
\item[ii)]Denote by $\hat{\M}:=\bigoplus_{j=1}^{N(N-1)/2} L_\infty(O_N, \M\overline{\ten}\M)$ and \[\delta:=\Big(\frac{N(N-1)}{2(N-2)}\Big)^{-1/2}(\nabla\ten \id_{\M\bten\M})\circ \pi\circ \Delta: \mathcal O(O_N^+) \to \hat{\M}\pl.\]
Then $(\mathcal O(O_N^+), \hat{\M}, \delta)$ is a derivation triple of $T_t$.
\item[iii)]$\nabla$ satisfies the $\Ric \ge 0$ relation: On $\mathcal O(O_N^+)$,
\[ \delta \circ T_t= (\id_{\oplus L_\infty(O_N)}\ten \id_{\M} \ten T_t ) \delta\]
\end{enumerate}
\end{prop}
\begin{proof}  In the following, we will abuse notation slightly and write $g_{ik}^k$ to mean both the scalar coefficient of the associated $O_N$-representation and also the coefficient {\it function}
$O_N \owns g \mapsto g_{ij}^k$.  Let $S_h$ be the left shift operator on $C(O_N)$ by $h \in O_N$, so that \[S_h g_{ij}^k=(hg)_{ij}^k=\sum_{1\le l\le n_k}h_{il}^kg_{lj}^k\pl.\]
Because $L$ is left invariant, \begin{align*}
L(g_{ij}^k)|_{g=h}=& S_h(L|_e (g_{ij}^k))=L(S_h(g_{ij}^k) )|_e
=\sum_{l}h_{il}^kL(g_{lj}^k)|_e,
\end{align*}
where $e$ is the identity element in $O_N$.
Therefore,
\begin{align*}
 \Big((L\ten \id)\circ  \pi\circ \Delta(u_{ij}^k)\Big)|_h
=&\Big(\sum_{l} (L\ten \id)\circ  \pi(u_{il}^k\ten u_{lj}^k)\Big)|_h
\\=&\Big(\sum_{l,m}(L\ten \id)(u_{im}^kg_{ml}^k\ten u_{lj}^k)\Big)|_h
\\=&\sum_{l,m}(u_{im}^k L(g_{ml}^k)|_h\ten u_{lj}^k)
\\=&\sum_{l,m,n}(u_{im}^kh_{mn}^k L(g_{nl}^k)|_e\ten u_{lj}^k)
\\=&\sum_{l,n}(\al_{h^{-1}}(u_{in}^k L(g_{nl}^k)|_e)\ten u_{lj}^k)
%\\=&\sum_{l,n}L({h^{-1}}(u_{in}^k g_{nl}^k)\ten u_{lj}^k)|_e
\\=&\al_{h^{-1}}\ten \id \Big(\big(L\ten \id)\circ  \pi\circ \Delta(u_{ij}^k)\big)|_e\Big)
\end{align*}
Note that the range of $\pi$ is  $\text{ran}(\pi)=\{f\in L_\infty(O_N,\M\overline{\ten}\M)\pl|\pl f(g)=\al_{g^{-1}}\ten \id (f(e)) \}$ and $E_\pi(f)=\int_{O_N} \al_g\ten \id(f)d g$.
Then $(L\ten \id)\circ  \pi\circ \Delta(u_{ij}^k)\in \text{ran} (\pi)$, and
\begin{align*}
 E_\Delta\circ E_\pi\circ(L\ten \id)\circ  \pi\circ \Delta(u_{ij}^k)
=&E_\Delta\Big((L\ten \id)\circ  \pi\circ \Delta(u_{ij}^k)|_e\Big)
\\=&E_\Delta\Big(\sum_{l,n}u_{in}^k L(g_{nl}^k)|_e\ten u_{lj}^k\Big)
\\=& \sum_{l}u_{il}^k L(g_{ll}^k)|_e\ten u_{lj}^k
\\=& \frac{L(\chi_k)|_e}{U_k(N)} u_{ij}^k
\end{align*}
where $g \mapsto \chi_k(g)=\sum_{l}g_{ll}^k$ is the character function of the representation of $O_N$ on $H^k$. Now it is sufficient to verify that there is a constant $C(N) \ne 0$ (depending only on $N$ and not $k$) so that
\[L(\chi_k)|_e=C(N)U_k'(N)\]
 By the recursive relation $ xU_k=U_{k-1}+U_{k+1}$
we have at $x=N$,
\[U_k(N)+NU_k'(N)=U_{k-1}'(N)+U_{k+1}'(N)\]
On the other hand, by the same fusion rule
\begin{align*}
\chi_1(g)\chi_k(g)=\chi_{k-1}(g)+\chi_{k-1}(g)
\end{align*} Applying the Casimir operator and evaluating at $g=e$, we thus obtain
\begin{align}\label{2}
L(\chi_1\chi_k)|_e=L(\chi_{k+1})|_e+L(\chi_{k-1})|_e
\end{align}
By the Leibniz rule, we have
\begin{align*}
L(\chi_1\chi_k)=&\sum_{j}X_j^2(\chi_1\chi_k)
 \\ =&\sum_{j}X_j\Big( (X_j\chi_1)\chi_k + \chi_1(X_j\chi_k)\Big)
\\ =&\sum_{j}(X_j^2\chi_1)\chi_k + (X_j\chi_1)(X_j\chi_k)+ \chi_1(X_j^2\chi_k)
 \\ =&(L\chi_1)\chi_k + \sum_j(X_j\chi_1)(X_j\chi_k)+ \chi_1(L\chi_k)
\end{align*}
Note that for $k=1$, $g\mapsto \chi_1(g) = \text{Tr}(g)$ is just the character of the fundamental representation of $O_N$, so at $g=e$,
\begin{align*}&X_j\chi_1|_{e}=\frac{d}{dt}\text{Tr}(\exp(tX_j))|_{t=0}=\text{Tr}(X_j)=0\pl,\end{align*}
where the last equality follows because $X_j\in \mathfrak{so}_N$ is skew-symmetric. Also,
\begin{align*} L(\chi_{1})|_e&=\sum_{j}X_jX_j\chi_1|_{e}=\sum_j\frac{\partial^2}{\partial s\pl \partial t} \pl \text{Tr}(\exp(sX_j)\exp(tX_j))|_{s,t=0}\\ &=\sum_j \text{Tr}(X_jX_j)=\frac{N(N-1)}{2(N-2)}\pl.
\end{align*}
Here the constant $\frac{N(N-1)}{2(N-2)}$ comes from the fact that $\text{dim}(\mathfrak{so}_N)=N(N-1)/2$ and $\{X_j\}$ is an orthonormal basis for negative Killing form of $\mathfrak{so}_N$,
\[ B(X,Y)=(N-2)\text{Tr}(XY)\pl.\]
Note that $\chi_1(e)=N,\chi_k(e)=\dim(H_k)=U_k(N)$. Then the \eqref{2} becomes
\begin{align*}
&\frac{N(N-1)}{2(N-2)}U_k(N) + NL(\chi_k)|_e=L(\chi_{k+1})|_e+L(\chi_{k-1})|_e
\end{align*}
This shows that the sequences  $\{\frac{N(N-1)}{2(N-2)}(L\chi_k)|_e\}$ and $\{U_k'(N)\}$ coincide because they satisfy the same recursive relation. This verifies i).

For ii),
we denote by $J=\pi\circ \Delta: \M \to L_\infty(O_N, \M \bar \otimes \M)$ the $*$-monomorphism and $E=E_\Delta\circ E_\pi$  the adjoint of $J$.
Let $\nabla$ be the gradient of $L$,
\[\nabla:C^\infty(O_N)\to \bigoplus_{j=1}^{N(N-1)/2} C^\infty(O_N)\pl, \pl \nabla(f)=\bigoplus_{j=1}^{N(N-1)/2} X_jf\pl.\]
It is clear that $\delta=\Big(\frac{N(N-1)}{2(N-2)}\Big)^{-1/2}(\nabla\ten \id)\circ J$ is a derivation. Note that by i), $\frac{N(N-1)}{2(N-2)}A=E\circ (L\ten \id)\circ J$. We verify the gradient form
\begin{align*} &\Gamma(x,y)=x^*Ay+(Ax)^*y-A(x^*y)\\ &=\Big(\frac{N(N-1)}{2(N-2)}\Big)^{-1}\Big(x^*E\circ(L\ten \id)\circ J(y)+E\circ(L\ten \id)\circ J(x^*)y-E\circ(L\ten \id)\circ J(x^*y))\Big)\\
&=\Big(\frac{N(N-1)}{2(N-2)}\Big)^{-1}E\Big( J(x^*)(L\ten \id)(J(y))+(L\ten \id)(J(x^*))J(y)+E((L\ten \id) J(x^*y)\Big)
\\
&=\Big(\frac{N(N-1)}{2(N-2)}\Big)^{-1}E\Big(\Gamma_{(L\ten \id)} (J(x),J(y))\Big)
\\
&=\Big(\frac{N(N-1)}{2(N-2)}\Big)^{-1}E\Big(E_{O_N}((\nabla \otimes \id) J(x))^*(\nabla \otimes \id) J(y))\Big)
\\
&=E(\delta(x)^*\delta(y))\pl.
\end{align*}
Finally, for iii) we verify that
\begin{align*}\delta\circ T_t &=
\Big(\frac{N(N-1)}{2(N-2)}\Big)^{-1/2}(\nabla\ten \id) \circ \pi\circ \Delta\circ T_t
\\
&=\Big(\frac{N(N-1)}{2(N-2)}\Big)^{-1/2}(\nabla\ten \id) \circ \pi\circ (\id_\M\ten T_t)\circ \Delta
\\
&=\Big(\frac{N(N-1)}{2(N-2)}\Big)^{-1/2}(\id_{\oplus L_\infty(O_N)}\ten \id_\M \ten T_t)\circ(\nabla\ten
\id) \circ \pi\circ \Delta
\\ &=(\id_{\oplus L_\infty(O_N)}\ten \id_\M \ten T_t)\circ \delta.
\end{align*} This completes the proof.
\end{proof}
We now turn to estimating the CB-return time.  We first recall some estimates of the growth of the dimensions $n_k=\dim k = U_k(N)$ and the eigenvalues $\la_k$ for the heat semigroup.
\begin{lemma}[Lemma 1.7 of \cite{FHLUZ17}]\label{dimen}
Denote by $n_k=U_k(N)$ and $\la_k=\frac{U_k'(N)}{U_k(N)}$. Then for $k \ge 0$,
\[n_k\le N^k, \frac{k}{N-2} \ge \la_k\ge \frac{k}{N}\pl.\]
\end{lemma}
%\begin{proof}
%We prove by induction. For $k=0,1$,  we have $n_0=1, \la_0=0$ and $n_1=N, \la_1=\frac{1}{N}$. Suppose for all $0\le k\le l$,
%\[n_k\le N^k, \la_{k-1}+\frac{1}{N}\le \la_k,  \la_k\ge \frac{k}{N}\pl.\]%
%Then
%\[n_{l+1}=Nn_{l}-n_{l-1}\le N^{l+1}\pl.\]
%Taking logarithm and then derivatives to the recursion formula, we have at $x=N$
%\begin{align*} \frac{1}{N}+\frac{U_k'(N)}{U_k(N)}=\frac{ U_{k+1}'(N)+U_{k-1}'(N)} {U_{k+1}(N)+U_{k-1}(N)}
%\end{align*}
%Thus
%\begin{align*}
%\la_k+\frac{1}{N}=\frac{ U_{k+1}'(N)+U_{k-1}'(N)}{U_{k+1}(N)+U_{k-1}(N)}=\frac{ \la_{k+1}U_{k+1}(N)+\la_{k-1}%U_{k-1}(N)}{U_{k+1}(N)+U_{k-1}(N)}
%\le \la_{k+1}%
%\end{align*}
%Thus $\la_{k+1}\ge \la_k+\frac{1}{N}\ge \frac{k+1}{N}$, which complete the induction
%\end{proof}
We are now ready to prove CLSI for $T_t$. We write $L_1(O_N^+):=L_1(L_\infty(O_N^+),\tau)$.
\begin{theorem}\label{orth-cbrt}
Let $T_t:L_\infty(O_N^+)\to L_\infty(O_N^+)$ be the heat semigroup defined as above. Let $E_\tau(a)=\tau(a)1$ be the conditional expectation onto the scalars. Then for $t> N\log N$,
\[\norm{T_t-E_\tau:L_1(O_N^+)\to L_\infty(O_N^+)}{cb}\le  \frac{2e^{(tN^{-1}-\log N)}-e^{2(tN^{-1}-\log N)}}{(1-e^{(tN^{-1}-\log N)})^2}\]
As a consequence, $T_t$ satisfies $\la$-CLSI for
$\la=\Big(4N \log \Big(\frac{N}{1 - \sqrt{\frac{2}{3}}}\Big)\Big)^{-1}$
\end{theorem}
\begin{proof}Denote $C(T_t)\in L_\infty(O_N^+)^{op}\overline{\ten} L_\infty(O_N^+)$ (resp. $C(E_\tau)$) as the Choi operator of $T_t$ (resp. $E_\tau$).
Note that $\{u_{ij}^k\}$ is an orthogonal set with $\tau((u_{ij}^k)^*u_{ij}^k)=\frac{1}{n_k}$. We have
\begin{align*}
&C(T_t)=\sum_{k\ge 0}e^{-\la_k t}n_k\sum_{1\le i,j\le  n_k} (u_{ij}^k)^{*op}\ten u_{ij}^k\pl \\
&C(E_\tau)=1^{op}\ten 1
\end{align*}
Then
\begin{align*}
\norm{T_t-E_\tau:L_1\to L_\infty}{cb}= &\norm{C(T_t)-C(E_\tau)}{L_\infty(O_N^+)^{op}\overline{\ten} L_\infty(O_N^+)}
\\= &\norm{\sum_{k\ge 1}e^{-\la_k t}n_k\sum_{1\le i,j\le  n_k} (u_{ij}^k)^{*op}\ten u_{ij}^k\pl}{L_\infty(O_N^+)^{op}\overline{\ten} L_\infty(O_N^+)} \\
\le&  \sum_{k\ge 1}e^{-\la_k t}n_k\norm{\sum_{1\le i,j\le  n_k} (u_{ij}^k)^{*op}\ten u_{ij}^k}{{L_\infty(O_N^+)^{op}\overline{\ten} L_\infty(O_N^+)}}\\
=&  \sum_{k\ge 1}e^{-\la_k t}n_k\norm{\sum_{1\le i,j\le  n_k} (Su_{ji}^k)^{op}\ten u_{ij}^k}{{L_\infty(O_N^+)^{op}\overline{\ten} L_\infty(O_N^+)}} \\
=& \sum_{k\ge 1}e^{-\la_k t}n_k\norm{\sum_{1\le j\le  n_k} (S\otimes \id) \Delta(u_{jj}^k)}{{L_\infty(O_N^+)^{op}\overline{\ten} L_\infty(O_N^+)}}
\end{align*}
But since $S:L_\infty(O_N^+) \to L_\infty(O_N^+)^{op}$ is a $\ast$-isomorphism and hence a complete isometry, we have
\begin{align*}
&\norm{\sum_{1\le j\le  n_k} (S\otimes \id) \Delta(u_{jj}^k)}{{L_\infty(O_N^+)^{op}\overline{\ten} L_\infty(O_N^+)}} \\
& = \norm{\sum_{j}u_{jj}^k}{L_\infty(O_N^+)} \\
&= \norm{\chi_k}{L_\infty(O_N^+)} \\
&= k+1.
\end{align*}
Note that in last equality, we have used the fact that $\chi_k = U_k(\chi_1)$ and the spectrum of $\chi_1$ relative to the Haar state is $[-2,2]$.    Combining the above and Lemma \ref{dimen}, we have
\begin{align*}
\norm{T_t-E_\tau:L_1\to L_\infty}{cb} & \le  \sum_{k\ge 1}e^{-\la_k t}(k+1)n_k\\
&\le \sum_{k\ge 1}(k+1)e^{- \frac{kt}{N}}N^{k} \\
& =  \sum_{k\ge1} (k+1) e^{k(-\frac{t}{N} + \log N)}
\end{align*}
Now, provided $t>N \log N$, we can let $r = e^{-\frac{t}{N} + \log N} < 1$ and get
\begin{align*}
\norm{T_t-E_\tau:L_1\to L_\infty}{cb}&\le  \sum_{k\ge1} (k+1) r^k  \\
&= \frac{d}{dr} \sum_{k \ge 1} r^{k+1}\\
&= \frac{d}{dr} \Big(\frac{r^2}{1-r}\Big) \\
&= \frac{2r-r^2}{(1-r)^2} \\
& = \frac{1}{2} \quad \big(\text{provided } r  = 1 - \sqrt{\frac{2}{3}}\big).
\end{align*}
This shows that $t_{cb}$ is given by \[e^{-\frac{t_{cb}}{N} + \log N} =r = 1 - \sqrt{\frac{2}{3}} \iff t_{cb} = N \log \Big(\frac{N}{1 - \sqrt{\frac{2}{3}}}\Big).\]
From this, we see that $T_t$ has $\la$-CLSI for $\la=\frac{1}{4t_{cb}}$, as claimed.
\end{proof}
\begin{rem}{\rm It was proved in \cite{BVY} that $T_t$ satisfies $\la$-LSI for $\la=2/N$. This is  asymptotically different with our CLSI-constant.}
\end{rem}

\subsubsection{Transference semigroup of small dimensional irreducible representation}
In this part we investigate the transference semigroup of small dimensional irreducible representation. Let $U\in B(H)\ten C(O_N^+)$ be a unitary representation on a finite dimensional Hilbert space $H$. Consider the (left) coaction induced by $U$,
\[\pi_U: B(H)\to  B(H)\ten C(O_N^+)\pl, \pl \pi(x)=U^*( x\ten 1)U\]
The transference semigroup $S_t:B(H)\to B(H)$ is defined by the following commuting diagram,
\begin{equation}\label{ccss}
 \begin{array}{ccc}  B(H)\ten L_{\infty}(O_N^+)\pl\pl &\overset{ T_t\ten \id }{\longrightarrow} & B(H)\ten L_{\infty}(O_N^+) \\
                    \uparrow \pi_U    & & \uparrow \pi_U  \\
                 B(H)\pl\pl &\overset{S_t}{\longrightarrow} & B(H)
                     \end{array} \pl .
                     \end{equation}
Here, the induced map $S_t$ is well-defined because $\text{ran}(T_t\ten \id\circ \pi_U)\subset \text{ran}(\pi_U)$. Indeed, let $E_U$ be the conditional expectation as adjoint of $\pi_U$,
\begin{align*}
\id\ten T_t( \pi_U(x))= &\id\ten T_t\ten \tau\left( \pi_U(x)\ten 1\right)
\\ =&\id\ten T_t\ten \tau \left( U^*_{12}(x\ten 1\ten 1)U_{12}\right)
\\=& \id\ten T_t\ten \tau \left( U_{13}U^*_{13}U^*_{12}(x\ten 1\ten 1)U_{12}U_{13} U_{13^*}\right)
\\= &\id\ten T_t\ten \tau  \left( U_{13} (\id\ten \D)(\pi(x)) U_{13^*}\right)
\\= &\id\ten \id\ten \tau  \left( U_{13} (\id\ten \D)(\id\ten T_t)(\pi(x)) U_{13^*}\right)
\\= &\id\ten \id\ten \tau  \left( U_{12}^*U_{12}U_{13} \id\ten \D(\id\ten T_t(\pi(x))) U_{13^*}U_{12}^*U_{12}\right)
\\= &\id\ten \id\ten \tau  \left( U_{12}^* \id\ten \D( U\id\ten T_t(\pi_U(x))U^*) U_{12}\right)
\\ =& U^*(\id\ten \tau\left(U(\id\ten T_t)(\pi(x))U^* ) \ten 1\right)U
\\ =& \pi_U\circ E_U \left(\id\ten T_t(\pi_U(x)\right)\pl.
\end{align*}
Thus $S_t$ is a quantum Markov semigroup on $B(H)$. By the interwine relation $\pi_U\circ S_t=(\id_{B(H)}\ten T_t)\circ \pi_U$, $S_t$ can be viewed as a sub-system of $\id_{B(H)}\ten T_t$. In particular, if $T_t$ has $\Ric\ge \la$ or $\la$-CLSI, so does $S_t$.

Now we consider the transference semigroup induced by the $1$-st irreducible representation
given by $U=\sum_{i,j}u_{ij}^1\ten e_{ij}\in C(O_N^+)\ten M_N$.
To calculate $S_t$ for $U^1$, we recall the fusion rule $1\ten 1=2\oplus 0$. That is, for each level-$1$ coefficients $a^1,b^1$,
\[ a^1b^1=\o{(a^1b^1)}+\tau(a^1b^1)1\]
where $\o{(a^1b^1)}\in \text{span}\{u^2_{ij} | 1\le i,j\le n_2\}$ is a level-$2$ coefficient.
Let \[\pi_1: M_N\to L_\infty(O_N^+)\ten M_N, \pi_1(x)=(U^1)^* (x\ten 1)U^1\] be the coaction of $U^1$. Then for the matrix unit $e_{sr}$ with $s\neq r$,
\begin{align*}
T_t(\pi_1(e_{sr}))=&T_t\Big((\sum_{ij} (u^1_{ij})^*\ten e_{ji})    (1\ten e_{sr})(\sum_{kl} u^1_{kl}e_{kl})\Big)
\\ =&T_t\Big(\sum_{j,l} (u^1_{sj})^*u^1_{rl}\ten e_{jl}\Big)
\\ =&T_t\Big(\sum_{j,l} \o{((u^1_{sj})^*u^1_{rl})}\ten e_{jl}\Big)
\\ =&e^{-\la_2 t}\sum_{j,l} \o{((u^1_{sj})^*u^1_{rl})}\ten e_{jl}
\\ =&\pi_1(e^{-\la_2 t} e_{sr})
\end{align*}
Here we use the fact $\tau((u^1_{sj})^*u^1_{rl})=0$ for any $j,l$ if $s\neq r$. For the matrix unit $e_{rr}$,
\begin{align*}
&T_t(\al(e_{rr}))\\=&T_t\Big((\sum_{ij} (u^1_{ij})^*\ten e_{ji})    (1\ten e_{rr})(\sum_{kl} u^1_{kl}e_{kl})\Big)
\\ =&T_t\Big(\sum_{j,l} (u^1_{rj})^*u^1_{rl}\ten e_{jl}\Big)
\\ =&T_t\Big(\sum_{j,l} \o{((u^1_{rj})^*u^1_{rl})}\ten e_{jl}+\sum_{l}\tau((u^1_{rl})^*u^1_{rl}) e_{ll}\Big)
\\ =&e^{-\la_2 t}\sum_{j,l} \o{((u^1_{rj})^*u^1_{rl})}\ten e_{jl}+\sum_{l}\tau((u^1_{rl})^*u^1_{rl}) e_{ll}
\\ =&e^{-\la_2 t}\Big(\sum_{j,l} \o{((u^1_{rj})^*u^1_{rl})}\ten e_{jl}+\sum_{l}\tau((u^1_{rl})^*u^1_{rl}) e_{ll}\Big)+(1-e^{-\la_2 t})\Big(\sum_{l}\tau((u^1_{rl})^*u^1_{rl}) e_{ll}\Big)
\\ =&e^{-\la_2 t}\pi_1(e_{rr})+(1-e^{-\la_2 t})\frac{1}{N}1
\\ =&\pi_1\Big(e^{-\la_2 t}e_{rr}+(1-e^{-\la_2 t})\frac{1}{N}1\Big)
\end{align*}
Thus $S_t:M_N\to M_N$ is exactly the depolarizing semigroup
\[ S_t(\rho)=e^{-\la_2 t}(\rho-tr(\rho)\frac{1}{N})+tr(\rho)\frac{1}{N}\pl.\]
which has $\la_2/2$-CGE by \cite[Section 3.3]{BGJ}. We will revisit the depolarizing semigroups in Section \ref{depo}.
%Recall the Chebyshev polynomial $u_2=xu_1-u_0=x^2-1$ and $u_2'= 2x$. Then
%\[\la_2=\frac{u_2'(N)}{u_2(N)}=\frac{2N}{N^2-1}\pl.\]
%Thus $S_t$ has curvature $\Ric\ge %\frac{N}{N^2-1}$. This implies optimal Ricci %curvature of $T_t$ is less than $\frac{N}{N^2-1}$.
%Let $\Delta_{O(N)}$ be the classical Casimir operator on $O_N$. Recall that from Proposition \ref{factor}, the generator of $T_t=e^{-At}$ on $O_N^+$ connecting to the classical Casimir $\Delta_{O_N}$ by
%\[\Delta_{O_N}\sim \frac{N(N-1)}{2(N-2)}A\]
%If we use the normalization to $\Delta_{O_N}$, then the transference semigroup of $\tilde{T}_t=e^{-\frac{N(N-1)}{2(N-2)}At}$ has Ricci curvature bounds by
%\[\frac{N(N-1)}{2(N-2)}\cdot\frac{N}{N^2-1}=\frac{N^2}{2(N+1)(N-2)}\]
%which converges to $1/2$ in the limit $N\to \infty$.

%On the other hand, for the classical case, the Ricci curvature lower bound of $SO(N)$. His %calculation shows that
%\[\min_{-K(X,X)=1} %Ric_{SO(N)}(X)=\frac{1}{8}\pl,\]
%where the Lie algebra element $X$ is normalized by negative Killing form $-K(X,Y)=(N-2)tr(XY)$. This implies the Ricci curvature lower bound for $\Delta_{O_N}$ is $1/8$ independent of $N$, which has a gap with the upper bound $\frac{N^2}{2(N+1)(N-2)}$ as well as the limit $1/2$ in the quantum case. Note that here the curvature lower bound of $\Delta_{O_N}$ does not pass to $\frac{N(N-1)}{2(N-2)}A$ because Proposition \ref{factor} is only a factorization of generator but not a transference of semigroup.

%{\bf Question:} Maybe I missed a constant $1/4$ so that they actually coincides in the limit $N\to \infty$?
\subsection{Quantum Automorphism Groups}  We briefly discuss here another class of examples of quantum groups given by the quantum automorphism groups of finite-dimensional C$^\ast$-algebras.  Let $B$ be a finite-dimensional C$^\ast$-algebra, and let $\psi: B \to \mathbb C$ be the canonical trace-state on $B$. Namely, $\psi$ is the restriction of the unique normalized trace on the endomorphism algebra $\text{End}(B)$, where $B \hookrightarrow\text{End}(B)$ via the left-regular representation of $B$.  Given any pair $(B,\psi)$, one can define the {\it quantum automorphism group of $(B,\psi)$}, which we denote by $G^+(B,\psi)$.  The construction goes as follows.  Put $H = L^2(B,\psi)$ and fix any orthonormal basis $(e_i)_i \subset H$ and let $e_{ij} \in B(H)$ be the corresponding matrix units.  Then $C^u(G^+(B,\psi))$ is the universal C$^\ast$-algebra with generators $u_{ij}$ subject to the following relations:
\begin{itemize}
\item[(1)] The matrix $u = \sum_{i,j} u_{ij} \otimes e_{ij} \in C^u(G^+(B,\psi)) \otimes B(H)$ is unitary.
\item[(2)] $(1 \otimes m)u_{12}u_{13} = u(1 \otimes m)$, where $m:H \otimes H \to H$ is the multiplication map on $B \cong H$.
\item[(3)] $u(1 \otimes \eta) = 1 \otimes \eta$, where $\eta:\mathbb C \to B \cong H;$ $\alpha \mapsto \alpha 1_B$ is the unit map.
\end{itemize}

We equip $C^u(G^+(B,\psi))$ with the coproduct given by the formula
\[\Delta(u_{ij}) = \sum_{k}u_{ik} \otimes u_{kj},\] and then the pair $G^+(B,\psi) = (C^u(G^+(B,\psi)), \Delta)$ is a compact quantum group.

\begin{exam}[Quantum Permutation Groups]{\rm
When $B = C(X)$, where $X$ is a finite set with $N$ elements, we have that $\psi$ corresponds to the uniform probability on $X$ and $G^+(B,\psi)$ is nothing other than the {\it quantum permutation group} \cite{wang98}, which is commonly denoted by $S_N^+$.  In this case the generators $u_{ij} \in C^u(S_N^+)$ satisfy the relations $u_{ij} = u_{ij}^* = u_{ij}^2$ and $\sum_k u_{ik} = 1 =  \sum_k u_{ki}$ for all $1 \le i \le N$.  When $N \le 3$, one can show that $C^u(S_N^+) \cong C(S_N)$, but when $N \ge 4$, $C^u(S_N^+)$ is infinite dimensional and non-commutative.}
\end{exam}

When $B$ is a finite-dimensional C$^\ast$-algebra with $\dim B \ge 4$, it turns out that the representation theory of $G^+(B,\psi)$ is very similar to that of $O_N^+$.  The irreducible representations $\Ir(G^+(B,\psi))$ are indexed by the non-negative integers $\mathbb{N}_0$ and the fusion rules are given by
\[k\ten m=|k-m|\oplus (|k-m|+1)\oplus \cdots \oplus (k+m-1)\oplus (k+m)\pl, k,m \in \mathbb{N} \pl.\]
Here, the label $0$ corresponds to the trivial representation and $1 \oplus 0$ corresponds to the fundamental representation $u = [u_{ij}] \in C^u(G^+(B,\psi)) \otimes B(H)$.
Recall that the Chebyshev polynomials of second kind are denoted by $(U_k)_{k \in \mathbb N_0}$ (cf. \eqref{chebychev}).  We then have the following analogue of Theorem \ref{orth-central} characterizing the generators of the central Markov semigroups on $L_\infty(G^+(B,\psi))$.  The following result (to the best of our knowledge) does not appear explicitly in the literature.  The the corresponding version for $S_N^+$ was proved in \cite[Theorem 10.10]{FKS16}, and the general statement can be proved using methods from \cite{Caspers18}, the classification of central states on $SU_q(2)$ and $SO_q(3)$ from \cite{DFY14}, and the fact that every $G^+(B,\psi)$ with $\dim B \ge 4$ is monoidally equivalent to $SO_q(3)$ with $q+q^{-1}  = \sqrt{\dim B}$ \cite{DV10}.  In the following, we let $v^k_{ij}$ denote a generic coefficient of the $k$th irreducible representation $v^k$  of $G^+(B,\psi)$.

\begin{theorem}{\cite[Corollary 10.3]{CFK14}} \label{qaut-central}
Let $d = \dim B \ge 4$.  Let $(\xi_k)_{k \in \mathbb N_0} \subset \mathbb C$ and define a central semigroup on $\mathcal O(G^+(B,\psi))$ via the formula \[T_t:\mathcal O(G^+(B,\psi)) \to \mathcal O(G^+(B,\psi)); \pl T_t(v_{ij}^k) = e^{-\xi_k t}v_{ij}^k.\]  Then $T_t$ extends uniquely to a symmetric central quantum Markov semigroup on $L_\infty(G^+(B,\psi))$ if and only if there is a constant $b \ge 0$ and a finite positive Borel measure $\nu$ supported on $[0,d]$ satisfying $\nu\{d\} = 0$, so that
\begin{align}\label{Hunt-formula}
\xi_k =\frac{1}{U_{2k}(\sqrt{d})}\Big( \frac{bU'_{2k}(\sqrt{d})}{2\sqrt{d}} +\int_{0}^d \frac{U_{2k}(\sqrt{d})-U_{2k}(\sqrt{x})}{x-d}d\nu(x)\Big).
\end{align}
\end{theorem}

By analogy with the case of $O_N^+$, we define the heat semigroup on $G^+(B,\psi)$ (with $d = \dim B \ge 4$) as $T_t=e^{-At}:L_\infty(G^+(B,\psi))\to L_\infty(G^+(B,\psi))$ given by
\[T_t(v_{ij}^k)=e^{-\xi_k t}v_{ij}^k\pl, \pl A(v_{ij}^k)=\xi_k v_{ij}^k\pl, \pl \xi_k=\frac{U'_{2k}(\sqrt{d})}{2\sqrt{d}U_{2k}(\sqrt{d})}\pl,\]
See \cite[Section 1.4]{CFK14}.
It follows from Theorem \ref{central} that $T_t$ satisfies $\Ric \ge 0$.  We also have the following estimates.
\begin{lemma}{\cite[Lemma 1.8]{CFK14}} \label{bds-qaut}
Let $m_k=\dim(v^k)=U_{2k}(\sqrt{d})$ and $\xi_k=\frac{U'_{2k}(\sqrt{d})}{2\sqrt{d}U_{2k}(\sqrt{d})}$. Then for $k\ge 0$
\[ m_k\le (d-1)^{k}, \quad \frac{k}{d} \le \xi_{k}\le \frac{k}{\sqrt{d}(\sqrt{d}-2)}\pl.\]
\end{lemma}

\begin{theorem} \label{qaut-cbrt}
Let $T_t:L_\infty(G^+(B,\psi))\to L_\infty(G^+(B,\psi))\pl, \pl T_t(v_{ij}^k)=e^{-\xi_k t}v_{ij}^k$ be the heat semigroup defined as above. Let $E_\tau(a)=\tau(a)1$ be the conditional expectation onto the scalars. Then for $t\ge d\log (d-1)$,
\[\norm{T_t-E_\tau:L_1(G^+(B,\psi))\to L_\infty(G^+(B,\psi))}{cb}\le\frac{4r-2r^2}{(1-r)^2} - \frac{r}{1-r}, \]
where $r = e^{-\frac{t}{d} + \log(d-1)}$.
As a consequence $T_t$ satisfies $\la$-CLSI for
$\la=\Big(4d \log \Big(\frac{3(d-1)}{ \Big(4 - \sqrt{13}\Big)}\Big)\Big)^{-1}$.
\end{theorem}
\begin{proof}The argument is analogous to the proof of Theorem \ref{orth-cbrt} so will sketch the main arguments.  Denote by $C(T_t)\in L_\infty(G^+(B,\psi))^{op}\overline{\ten} L_\infty(G^+(B,\psi))$ (resp. $C(E_\tau)$) the Choi operator of $T_t$ (resp. $E_\tau$).
Note that $\{v_{ij}^k\}$ is an orthogonal basis with $\tau((v_{ij}^k)^*v_{ij}^k)=\frac{1}{m_k}$. We have
\begin{align*}
&C(T_t)=\sum_{k\ge 0}e^{-\xi_k t}m_k\sum_{1\le i,j\le  m_k} (v_{ij}^k)^{*op}\ten v_{ij}^k\pl, \\
&C(E_\tau)=1^{op}\ten 1
\end{align*}
Then
\begin{align*}
\norm{T_t-E_\tau:L_1\to L_\infty}{cb}= &\norm{C(T_t)-C(E_\tau)}{L_\infty(G^+(B,\psi))^{op}\overline{\ten} L_\infty(G^+(B,\psi))}
\\= &\norm{\sum_{k\ge 1}e^{-\xi_k t}m_k\sum_{1\le i,j\le  m_k} (v_{ij}^k)^{*op}\ten v_{ij}^k\pl}{L_\infty(G^+(B,\psi))^{op}\overline{\ten} L_\infty(G^+(B,\psi))}\\
\le& \sum_{k\ge 1}e^{-\xi_k t}m_k\norm{\sum_{1 \le i,j \le m_k}(v_{ij}^k)^{*op}\ten v_{ij}^k\pl}{L_\infty(G^+(B,\psi))^{op}\overline{\ten} L_\infty(G^+(B,\psi))} \\
=&  \sum_{k\ge 1}e^{-\xi_k t}m_k\norm{\chi_k}{L_\infty(G^+(B,\psi))} \\
=& \sum_{k\ge 1}e^{-\xi_k t}m_k(2k+1) \\
\le& \sum_{k \ge 1} (2k+1)e^{-\frac{kt}{d} + k\log(d-1)}.
\end{align*}
Here, we used Lemma \ref{bds-qaut} and the equality $\norm{\chi_k}{L_\infty(G^+(B,\psi))} = \sup_{t \in [0,4]}|U_{2k}(\sqrt{t})|$ \cite{Brannan13}.  We then get, for $t > d \log (d-1)$ and $r = e^{-\frac{t}{d} + \log(d-1)},$

\begin{align*}
\norm{T_t-E_\tau:L_1\to L_\infty}{cb}&\le  \sum_{k\ge1} (2k+1) r^k  \\
&= \frac{4r-2r^2}{(1-r)^2} - \frac{r}{1-r} \\
& = \frac{1}{2} \quad \Big(\text{provided } r  = \frac{1}{3}\Big(4 - \sqrt{13}\Big)\Big).
\end{align*}
This shows that $t_{cb}$ is given by \[e^{-\frac{t_{cb}}{d} + \log (d-1)} =r =  \frac{1}{3}\Big(4 - \sqrt{13}\Big) \iff t_{cb} = d \log \Big(\frac{d-1}{ \frac{1}{3}\Big(4 - \sqrt{13}\Big)}\Big).\]
By Theorem \ref{CLSI3}, we see that $T_t$ has $\la$-CLSI for $\la=\frac{4}{t_{cb}},$ as claimed.
\end{proof}

\begin{rem}{\rm It is desired to have a concrete derivation triple for the heat semigroup on $L_\infty(G^+(B,\psi))$. Nevertheless, for a special case, one can show that the heat semigroup on quantum permutation group $L_\infty(S_N^+)$ does not admit factorization through any classical Markov semigroup on $l_\infty(S_N)$ as in Proposition \ref{factor} .
}
\end{rem}

\section{Tensorization and Free product}
In this section, we discuss tensorization and free product of CLSI and geometric Ricci curvature bound. The similar discussion for CGE is in \cite{WZ}.
\subsection{Commutting Semigroup}
Let $T_t,S_t: \M \to \M$ be two symmetric quantum Markov semigroups and $A$ (resp. $B$) be the generator of $T_t$ (resp. $S_t$). We say $T_t$ and $S_t$ are commuting if $T_t\circ S_s= S_s\circ T_t$ for any $s,t\ge 0$. For commuting $T_t$ and $S_t$,
$S_t\circ T_t$ is again an symmetric quantum Markov semigroup because
\[(S_s\circ T_s)\circ (S_t \circ T_t)=(S_s\circ S_t)\circ (T_s\circ T_t)=S_{s+t}\circ T_{s+t}\pl.\]
Let $\N_T$ (resp. $\N_T$) be the fixed point algebra of $S_t$ (resp. $\N_S$).
Then $\N=\N_S\cap\N_T$ is the fixed point subalgebra of $S_t \circ T_t$
We write $E_T,E_S$ and $E$ as the condition expectation respectively onto $\N_T$, $\N_S$ and $\N$. The following lemma is inspired from  \cite[Corollary 4.2]{LaRacuente19} by LaRacuente.
\begin{prop}\label{commuting}
Let $T_t, S_t : \M \to \M$ be two symmetric quantum Markov semigroups. Suppose $T_t$ and $S_t$ are commuting. Then $S_t\circ T_t$ is a symmetric quantum Markov semigroup.
If in additional, both $T_t$ and $S_t$ satisfies $\la$-MLSI (resp. $\la$-CLSI).
\begin{enumerate}
\item[i)] $E_S\circ E_T=E_T\circ E_S=E$ forms a commuting square.
\item[ii)] $T_t\circ S_t$ satisfies $\la$-MLSI (resp. $\la$-CLSI).
\end{enumerate}
\end{prop}
\begin{proof}Because $T_t$ and $S_t$ satisfies $\la$-MLSI, we have the mixing time estimate that for any density operator $\rho$ with finite entropy $H(\rho)<\infty$ (see \cite{bardet}),
\[ \lim_{t\to \infty}\norm{T_t(\rho)-E_T(\rho)}{1}\le \lim_{t\to \infty}\sqrt{2D(T_t(\rho)||E_T(\rho))}\le \lim_{t\to \infty}\sqrt{2e^{-2\la t}D(\rho||E(\rho))}=0\pl.\]
and similar $\displaystyle\lim_{t\to \infty}\norm{S_t(\rho)-E_S(\rho)}{1}=0$
This implies for any $x\in L_1\cap L_\infty(\M)$
\[T_t\circ E_S(x)=\lim_{s\to \infty}T_t\circ S_s(x) =\lim_{s\to \infty} S_s\circ T_t(x) =E_S(x)\circ T_t\pl,\]
By continuity the same equality extends to $\M$.
Thus we have $T_t\circ E_S=E_S\circ T_t$ and by the same argument for $T_t\to E_T$, we have
$E=E_T\circ E_S=E_S\circ E_T$ forms a commuting square. It follows that
\begin{align*}D(\rho||E(\rho))=&H(\rho)-H(E(\rho))=H(\rho)-H(E_S(\rho))+H(E_S(\rho))-H(E(\rho)) \\ =&D(\rho||E_S(\rho))+D(E_S(\rho)||E(\rho))\pl.\end{align*}
By data processing inequality,
\begin{align*}
D(T_t\circ S_t(\rho)||E(\rho))&=D(S_t(T_t\rho)||E_S(T_t\rho))+D(T_t( E_S\rho)||E_T\circ E_S(\rho))
\\ &\le e^{-2\la t}D(T_t(\rho)||T_t(E_S\rho))+e^{-2\la t}D( E_S(\rho)||E(\rho))
\\ &\le e^{-2\la t}D(\rho||E_S(\rho))+e^{-2\la t}D( E_S(\rho)||E(\rho))
\\ &=e^{-2\la t}D(\rho)||E(\rho))\pl.
\end{align*}
which implies that $T_t\circ S_t$ has $\la$-MLSI.
The same argument for $T_t\ten \id_\R,S_t\ten \id_\R: \M\overline{\ten} \R \to \M\overline{\ten} \R$ yield the assertion for CLSI.
\end{proof}

\subsection{Tensor product semigroup}

Let $T_t:\M_1\to \M_1$ and $S_t:\M_2\to \M_2$ be two symmetric quantum Markov semigroups. The tensor product $T_t\ten S_t: \M_1\bten \M_2\to \M_1\bten \M_2$ is again a symmetric Markov semigroup.

\begin{cor}
Let $T_t:\M_1\to \M_1$ and $S_t:\M_2\to \M_2$ be two quantum Markov semigroups. If $T_t$ and $S_t$ satisfy $\la$-CLSI, $T_t\ten S_t$ satisfy $\la$-CLSI.
\end{cor}
\begin{proof}By definition of CLSI, both $T_t\ten \id_{\M_2}$ and $ \id_{\M_1}\ten S_t$ satisfy $\la$-CLSI. The assertion follows from Proposition \ref{commuting}.
\end{proof}

We shall now discuss the tensorization of $\Ric$.
Let $A$ (resp. $B$) be the generator of $T_t$ (resp. $S_t$).
Let $L$ be the generator of $T_t\ten S_t$. Then $T$ is an closed extension of $L=A\ten \id +\id \ten B$ and $\dom(A)\ten \dom(B)\subset \dom(L)$. Namely, for $x\in \dom(A),y\in \dom(B)$, $L(x\ten y)= Ax\ten y+ x\ten By$. The gradient form of $T_t\ten S_t$ is
\[ \Gamma(x\ten y, x\ten y)=\Gamma_A(x,x)\ten y^*y+x^*x\ten \Gamma_B(y,y)\pl.\]
where $\Gamma_A$ (resp. $\Gamma_B$) is the gradient form of $A$ (resp. $B$).%Recall that a derivation triple of $(\A,\hat{\M},\delta)$ of $T_t:\M\to\M$ is given by a closed symmetric derivation $\delta:{\rm dom}(A^{1/2})\to L_2(\hat{\M})$ such that
%\begin{align}E_\M(\delta(x)^*\delta(y))=\Gamma(x,y)\label{deltagamma1}\end{align}
%where $\hat{\M}$ is a finite von Neumann algebra $\hat{\M}$ containing $\M$ with induced trace and $E_\M:\hat{\M} \to \M$ is the conditional expectation. $\A$ is a $w^*$-dense subalgebra $\A\subset \M$ such that $\A\subset \dom(A^{1/2})\pl, T_t(\A)\subset \A$.
%We say the derivation $\delta$ has mean zero property if $E_\M(\delta(x))=0$ for all $x\in \dom(A^{1/2})$.

\begin{lemma}\label{tende}Let $T_t:\M_1\to \M_1$ and $S_t:\M_2\to \M_2$ be two symmetric quantum Markov semigroups. Let $(\A,\hat{\M}_1,\delta_1)$ (resp. $(\B,\hat{\M}_2,\delta_2)$) be a derivation triple of $T_t$ (resp. $S_t$) with mean zero property. Define the derivation \[\delta:\A\ten \B\to L_2(\hat{\M}_1\bten \hat{\M}_2)\pl, \delta(x\ten y)=\delta_1(x)\ten y+ x\ten \delta_2(y)\pl. \]
Then $(\A\ten \B, \hat{\M}_1\bten \hat{\M}_2, \delta)$ is derivation triple of $T_t\ten S_t$.
\end{lemma}
\begin{proof}Let $E_1:\hat{\M}_1\to \M_1$ and $E_2:\hat{\M}_2\to \M_2$ the conditional expectation. Then $E=E_1\ten E_2$ is the conditional expectation from $\hat{\M}_1\bten \hat{\M}_2$ to $\M_1\bten\M_2$. It is clear that
$\delta$ is a $*$-preserving derivation and mean zero $E(\delta(x\ten y))=E_1(\delta_1(x))\ten y+x\ten E_2(\delta_2(y))=0$. For $x\in \A, y\in \B$, we have
\begin{align*}&E(\delta(x\ten y)^* \delta(x\ten y))\\ =&
E(\delta_1(x)^*\delta_1(x)\ten y^*y+ \delta_1(x)^*x\ten y^*\delta_2(y)+x^*\delta_1(x)\ten y^*\delta_2(y)+x^*x\ten \delta_2(y)^*\delta_2(y))
\\ =&
E_1(\delta_1(x)^*\delta_1(x))\ten y^*y+ E_1(\delta_1(x)^*x)\ten E_2(y^*\delta_2(y))\\ &+E_1(x^*\delta_1(x))\ten E_2(y^*\delta_2(y))+x^*x\ten E_2(\delta_2(y)^*\delta_2(y))
\\ =&
\Gamma_{1}(x)\ten y^*y+ E_1(\delta_1(x)^*)x\ten y^*E_2(\delta_2(y))\\ &+x^*E_1(\delta_1(x))\ten E_2(\delta_2(y)^*)y+x^*x\ten x^*x\ten \Gamma_{2}(y,y)
\\ =& \Gamma_1(x, x)\ten y^*y+x^*x\ten \Gamma_2(y, y)
\\ =& \Gamma(x\ten y, x\ten y),
\end{align*}
which coincides with the gradient form of $T_t\ten S_t$. Here in the second last step we used the mean zero property $E_1(\delta_1(x))=E_2(\delta_2(x))=0$. It follows that for any $\xi\in \A\ten \B$,
\[\lan\delta(\xi),\delta(\xi)\ran_{L_2(\hat{\M}_1\overline{\ten} \hat{\M}_2)}
= \lan\xi, (A\ten \id +\id \ten B)\xi\ran_{L_2(\M_1\overline{\ten} \M_2)}=\lan\xi, L\xi\ran_{L_2(\M_1\overline{\ten} \M_2)}\pl.
\]
Note that $\A\ten \B$ is a subalgebra satisfying
\begin{align*} &T_t\ten S_t(\A\ten \B)\subset T_t(\A)\ten S_t(\B)\subset \A\ten \B\pl, \\ &\A\ten \B\subset \dom(A^{1/2})\ten \dom(B^{1/2})\subset \dom(L^{1/2}) \pl. \qedhere\end{align*}
Moreover, $\A\ten \B$ is dense in $\A_\E\ten \B_\E$ with respect to the graph norm, which is a core for $\dom(L^{1/2})$. Thus $\delta$ admits an closed extension $\bar{\delta}$ such that $\delta^*\bar{\delta}=L$.  That completes proof.
\end{proof}
%Let $T_t:\M\to \M$ be a symmetric quantum Markov semigroup and let $(\A,\hat{\M},\delta)$ be a derivation triple for $T_t$.
%Recall that we say $T_t:\M\to \M$ satisfies $\Ric\ge \la$ for $\la\in \mathbb{R}$ if there exists a derivation triple $(\A,\hat{\M},\delta)$ for $T_t$ and a symmetric Markov semigroup $\hat{T}_t:\hat{\M}\to \hat{\M}$ with generator $\hat{A}$ such that
%\begin{enumerate}
%\item[i)] $\hat{T}_t|_{\M}=T_t$ for $t\ge 0$.
%\item[ii)] $\delta(\A_0)\in \text{dom}(\hat{A})$ and there exists a $\A_\E$-bimodule map $:\Omega_\delta\to L_2(\hat{\M})$ such that for any $x\in\A_0$,
%    \[ \hat{A}\delta(x)-\delta A(x)= \text{Ric}(\delta(x))\pl.\]
%\item[iii)] for any $\xi\in \Omega_\delta$,
    %\[    \lan\xi,\text{Ric}(\xi)\ran \ge \la \lan \xi, \xi\ran \pl.\]
%Here $\Omega_\delta=\A_\E\delta(\A_\E)$.
%\end{enumerate}
\begin{prop}\label{tensoralg}
Let $T_t : \M_1 \to \M_1$ and $S_t : \M_2 \to \M_2$ be two symmetric quantum Markov semigroups.
\begin{enumerate}\item[i)]If both $S_t,T_t$ has $\Ric\ge \la$, $S_t\ten T_t$ has $\Ric\ge \la$;
\item[ii)]If both $S_t,T_t$ has $\la$-$\Ric$, $S_t\ten T_t$ has $\la$-$\Ric$.
\end{enumerate}
\end{prop}
\begin{proof}Let $(\A,\hat{\M_1},\delta_1)$  (resp. $(\B,\hat{\M_2},\delta_2)$ ) be a  derivation triple of $T_t$ (resp. $S_t$) and $\hat{T}_t=e^{-\hat{A}t}:\hat{\M}_1\to \hat{\M}_1$ and $\hat{S}_t=e^{-\hat{B}t}:\hat{\M}_2\to \hat{\M}_2$ be the extension Markov semigroups giving $\Ric\ge \la$ respectively. That is, $\hat{T}_t|_{\M_1}=T_t$ and $\hat{S}_t|_{\M_2}=S_t$
\begin{align*} &\hat{A}\delta_1(x)-\delta_1 A(x)= \ric_A(\delta_1(x))\pl,  \pl x\in \A_0\\ &\hat{B}\delta_2(y)-\delta_2B(y)= \ric_B(\delta_2(y))\pl, \pl y\in \B_0.
\end{align*}
with bimodule Ricci operator $\ric_A\ge \la$ and $\ric_B\ge \la$. Here $A,B$ (resp. $\hat{A},\hat{B}$) are the generator of $T_t,S_t$ (resp, $\hat{T}_t,\hat{S}_t$). In particular $\hat{A}|_{\dom(A)}=A$ and $\hat{B}|_{\dom(B)}=B$.

We take the derivation triple $(\A\ten \B, \hat{\M}_1\bten\hat{\M}_2, \delta)$ for $T_t\ten S_t$ in Lemma \ref{tende}. Consider the tensor product semigroup $\hat{T}_t\ten \hat{S}_t: \hat{\M}_1\bten\hat{\M}_2 \to \hat{\M}_1\bten\hat{\M}_2$. We have $\hat{T}_t\ten \hat{S}_t|_{\M_1\bten \M_2}=T_t\ten S_t$ and its generator is $\hat{L}:=\hat{A}\ten \id+\id\ten \hat{B}$. For any $x\in \A_0,y\in \B_0$,
\begin{align*}
&\hat{L}\delta(x\ten y)-\delta({L}(x\ten y))\\=&
 (  \hat{A}\delta_1(x)\ten y + \delta_1(x)\ten By+ Ax\ten \delta_2(y)+x\ten \hat{B}\delta_2(y))
 \\&-(  \delta_1(Ax)\ten y + \delta_1(x)\ten By+ Ax\ten \delta_2(y)+x\ten \delta_2(By))
 \\ =& (\hat{A}\delta_1(x)-\delta_1(Ax))\ten y+x\ten (\hat{B}\delta_2(y)-\delta_2(By))
 \\ =& \ric_A(\delta_1(x))\ten y+x\ten \ric_B (\delta_2(y))
\end{align*}
Note that $\Omega_\delta=\overline{(\A\ten \B)\delta(\A\ten \B)}\subset L_2(\hat{\M}_1\overline{\ten} \hat{\M}_2)$ and
\[(\A\ten \B)\delta(\A\ten \B)=\A\ten \B\delta_2(\B)+\A\delta_1(\A)\ten \B\pl.\]
Moreover, by the mean zero property, $E_1(\A\delta_1(\A))=E_2(\B\delta_2(\B))=0$. Indeed, for any $\xi_1=x_1\delta_1(y_1)\in \Omega_{\delta_1}$,
\[ E_1(\xi_1)=E_1(x_1\delta_1(y_1))=x_1E_1(\delta_1(y_1))=0\pl.\]
Then $\A\ten \B\delta_2(\B)$ and $\A\delta_1(\A)\ten \B$ are mutually orthogonal in $L_2(\hat{\M}_1\overline{\ten} \hat{\M}_2)$. We can define
\[\ric:\Omega_\delta\to L_2(\hat{\M}_1\overline{\ten} \hat{\M}_2)\pl, \ric=(\ric_A\ten \id) \oplus(\id \ten \ric_B)\pl.\]
It is clear that $\ric$ is a $\A\ten \B$-bimodule operator. For any $\xi\in \Omega_\delta$, we can write $\xi=\xi_1+\xi_2$ with $\xi_1\in \overline{\Omega_{\delta_1}\ten \B}$ and $\xi_2\in \overline{\A\ten \Omega_{\delta_2}}$. Then
\begin{align*} \lan \xi, \ric(\xi)\ran=& \lan \xi_1+\xi_2, \ric_A\ten \id(\xi_1)+\id \ten\ric_B(\xi_2)\ran=
\\ =& \lan \xi_1, \ric_A\ten\id (\xi_1)\ran+\lan \xi_2, \id\ten \ric_B (\xi_2)\ran=
 \\ \ge& \la \norm{ \xi_1}{2}+\la \norm{\xi_2}{2} =\la\norm{\xi}{2}\pl.\end{align*}
That completes the proof.
\end{proof}

\subsection{Free product semigroup}
Let $\M_1,\M_2$ be two finite von Neumann algebra and $\N\subset \M_1, \N\subset \M_2$ be a common subalgebra. We refer to \cite{VDN} for the definition amalgamated free product $\M_1*_\N\M_2$. Denote $E_\N: \M_i\to \N\pl, \pl i=1,2$
as the conditional expectation onto $\N$ and write
 $ \o{\M_i}=\{ a\in \M_i\pl |\pl  E_\N(a)=0\}$
as the mean zero part. It was proved in \cite[Theorem 3.1]{boca1} that for two $\N$-bimodule UCP maps $T_1:\M_1\to \M_1$ and $T_2:\M_2\to \M_2$, the free product map $T_1*T_2:\M_1*_\N\M_2 \to \M_1*_\N\M_2$,
\[ T_1*T_2(a_1a_2\cdots a_n)=T_{i_1}(a_1)T_{i_2}(a_2)\cdots T_{i_n}(a_n)\pl ,  a_k\in \o{\M_{i_k}}, i_1\neq i_2\neq \cdots \neq i_n\pl. \]
is again UCP.

Let $T_{1,t}=e^{-At}:\M_1\to \M_1$ and $T_{2,t}=e^{-Bt}:\M_2\to \M_2$ be two symmetric Markov semigroup with fixed-point algebra $\N_1$ and $\N_2$ respectively. Suppose $\N\subset\N_1,\N_2$ as a common subalgebra. Then the free product map $T_t=T_{1,t}*T_{2,t}:\M_1*_\N\M_2\to \M_1*_\N\M_2$ is a symmetric quantum Markov semigroup that for any $a_i\in \M_{i_n}$
 \[ T_{1,t}*T_{2,t}(a_1\cdots a_n)=T_{i_1,t}(a_1)T_{i_2,t}(a_2)\cdots T_{i_n,t}(a_n) \pl, a_k\in \M_{i_k}\pl, i_1\neq i_2\neq\cdots \neq i_n\]
 This map is well-defined because $T_{i,t}(x)=x$ for $x\in \N, i=1,2$.

\begin{defi} We say a Markov semigroup $T_t:\M\to \M$ satisfies $\la$-free logarithmic Sobolev inequality ($\la$-FLSI) for $\la\in \R$ if for any finite von Neumann algebra $\R$ with $\N\subset \mathcal{R}$, $T_t*\id_\mathcal{R}: \M*_{\N}\mathcal{R}\to \M*_{\N}\mathcal{R}$ has $\la$-MLSI with respect to $\mathcal{R}\simeq \N*_\N\mathcal{R}\subset \M*_{\N}\mathcal{R}$.\\
\end{defi}
\begin{prop}
Let $T_t:\M_1\to \M_1$ and $S_t:\M_2\to \M_2$ be two symmetric Markov semigroup with same fixed-point subalgebra $\N$. If $T_t,S_t$ satisfies $\la$-FLSI, $T_t*S_t$ satisfies $\la$-FLSI.
\end{prop}
\begin{proof}Let $\mathcal{R}$ be an arbitrary finite von Neumann algebra. We have
\[T_t*S_t*\id_\mathcal{R}: \M_1*_\N\M_2*_\N\mathcal{R}\to \M_1*_\N\M_2*_\N\mathcal{R}\pl \]
is a symmetric Markov semigroup satisfying
\[ T_t*S_t*\id=(T_t*\id_{\M_2}*\id_\mathcal{R})\circ (\id_{\M_1}*S_t*\id_\mathcal{R})= (\id_{\M_1}*S_t*\id_\mathcal{R})\circ(T_t*\id_{\M_2}*\id_\mathcal{R})\pl.\]
Then the assertion follows from applying Proposition \ref{commuting} for
By assumption of FLSI,
\[T_t*\id_{\M_2}*\id_\mathcal{R}=T_t*\id_{\M_2*_\N\mathcal{R}}\pl, \pl (\id_{\M_1}*S_t*\id_\mathcal{R})=S_t*\id_{\M_1*_\N\mathcal{R}}\] satisfy $\la$-MLSI. Then the assertion follows from applying Proposition \ref{commuting} to the above to semigroup on $\M_1*_\N\M_2*_\N\mathcal{R}$.
\end{proof}

 Let $T_{1,t}=e^{-A_1t}:\M_1\to \M_1$ and $T_{2,t}=e^{-A_2t}:\M_2\to \M_2$ be two symmetric Markov semigroup with fixed-point algebra $\N_1$ and $\N_2$ respectively. Let $\N\subset \N_1,\N_2$ be a common subalgebra of $\N_1$ and $\N_2$. The generator of free product semigroup $T_t=T_{1,t}*T_{2,t}$ on $\M_1*_\N\M_2$ is
 \begin{align*}
 &La=A_1a\pl, Lb=A_2b\pl, a\in \dom(A_1), b\in \dom(A_2)
 \\ &L(a_1\cdots a_n)=\sum_{k}a_1\cdots La_k\cdots a_n
\end{align*}
where $a_k\in \dom(A_{i_k})\cap \o{\M_{i_k}}, i_1\neq i_2\neq\cdots \neq i_n$.
%Let $\Gamma_1$ (resp $\Gamma_2$) be the gradient form of $T_t$ (resp. $T_{2,t}$).

\begin{prop}\label{freede}Let $(\A_1,\hat{\M}_1,\delta_1)$ (resp. $(\A_2,\hat{\M}_2,\delta_2)$) be a derivation triple of $T_{1,t}$ (resp. $T_{2,t}$). Denote $\A_1*_\N \A_2$ as the algebraic free product. Define the closable derivation $\delta: \A_1*_\N \A_2\to  L_2(\hat{\M}_1*_\N\hat{\M}_2)$
\begin{align*}&\delta(a)=\delta_1(a)\pl, \delta(b)=\delta_2(b), a\in \A_1, b\in \A_2
\\ & \delta(a_1\cdots a_n)=\sum_{k=1}^na_1\cdots \delta_{i_k}(a_k)\cdots a_n\pl, a_k\in \A_{i_k}\pl.
\end{align*}
Then $(\A_1*_\N \A_2, \hat{\M}_1*_\N\hat{\M}_2, \delta)$
is a derivation triple for $T_t=T_{1,t}*T_{2,t}$.
\end{prop}
\begin{proof}
Let $E_i:\M_{i}\to\N_i,i=1,2$ be the conditional expectation to the fixed point subalgebra. Then $T_{1,t}* T_{2,t}$ has fixed point subalgebra $\N_1*_\N\N_2$ and conditional expectation $E=E_1*E_2$.
It is clear from the definition that $\delta$ is $*$-preserving and satisfy Leibniz rule. For $a_k\in \o{\A}_{i_k}, i_1\neq i_2\neq\cdots \neq i_n$,
\[ E(\delta(a_1\cdots a_n))=\sum_{k=1}^nE_{1}*E_2(a_1\cdots \delta_{i_k}(a_k)\cdots a_n)=\sum_{k=1}^nE_{i_1}(a_1)\cdots  E_{i_k}(\delta_{i_k}a_k)\cdots E_{i_n}(a_n)=0\pl.\]
To verify the $\Gamma_L(v,w)=E(\delta(v)^*\delta(w))$. It suffices to consider the free word $v=b_1\cdots b_m$ and $w=a_1\cdots a_n$, where $a_k\in \o{\A}_{i_k}$ and $b_l\in \o{\A}_{j_l}$. We argue it by induction. Denote $\o{a}=a-E_\N(a)\pl$ for the mean zero part. The initial step is $v=b$ with $b\in \A_j$ and $w=a_1\cdots a_n$,
\begin{align*}
2\Gamma_L(b,w)&=b^*Lw+(Lb)^*w-L(b^*w)
\\&=\sum_{k=1}b^*a\cdots La_k\cdots a_n+(Lb^*)a_1\cdots a_n-L(b^*a_1\cdots a_n)
%\\&=\sum_{k=1}b^*a\cdots La_k\cdots a_n+(Lb^*)a_1\cdots a_n-L(b^*a_1\cdots a_n)
\end{align*}
For the last term we have two cases: if $j=i_1$,
\begin{align*}
L(b^*a_1\cdots a_n)=&L\Big(\o{(b^*a_1)}a_2\cdots a_n+ E_\N(b^*a_1)a_2\cdots a_n\Big)
\\ =& L\o{(b^*a_1)}a_2\cdots a_n+\o{(b^*a_1)}\sum_{k=2}^na_2 \cdots La_k\cdots a_n+E_\N(b^*a_1)\sum_{k=2}^na_2\cdots La_k\cdots  a_n
\\ =&         L\o{(b^*a_1)}a_2\cdots a_n+(b^*a_1)\sum_{k=2}^na_2 \cdots La_k\cdots a_n
\end{align*}
If $j\neq i_1$.
\begin{align*}
L(b^*a_1\cdots a_n)=&(Lb^*)a_1a_2\cdots a_n+ b^*\sum_{k}a_1\cdots La_k\cdots a_n
\end{align*}
In total, $L(b^*a_1\cdots a_n)$ equals
\[ \begin{cases}
                         L(b^*a_1)a_2\cdots a_n+(b^*a_1)\sum_{k=2}a_2\cdots La_k\cdots a_n, & \mbox{if } j=i_1 \\
                         (Lb^*)a_1a_2\cdots a_n+b^*(La_1) a_2\cdots a_n+ (b^*a_1)\sum_{k=2}a_2\cdots La_k\cdots a_n, & \mbox{otherwise} .
                       \end{cases}
\]
For both case, the last term cancels in $\Gamma_L(b,w)$. We have
\begin{align*}
\Gamma_L(b,a_1a_2\cdots a_n)&=\begin{cases}
                         \Gamma_{j}(b, a_1)a_2\cdots a_n, & \mbox{if } j=i_1 \\
                         0 & \mbox{otherwise} .
                       \end{cases}
\end{align*}
Let $E_{\M_i}:\hat{\M}_i\to \M_i, i=1,2$ be the conditional expectation.
Now we calculate that
\begin{align*}
E_\M(\delta(b)^*\delta(w))=E_\M\Big( \delta_j(b)^*\sum_{k=1}a_1\cdots \delta_{i_k}(a_k)\cdots  a_n\Big)\end{align*}
where $E_\M=E_{\M_1}*E_{\M_2}$.
Because $E_{\M_j}(\delta_j(a))=0$ if $a\in \A_j$,  we know the only nonzero case is $i_1=j$ and
\begin{align*}
E_\M(\delta(b)^*\delta(w))=&E_\M\Big( \delta_j(b)^*\delta_j(a_1)\cdots  a_n\Big)\\ =&E_{\M_j}(\delta_j(b)^*\delta_j(a_1))E_{\M_{i_2}}(a_2)\cdots E_{\M_{i_n}}(a_n)
\\ =&E_{\M_j}(\delta_j(b)^*\delta_j(a_1))a_2\cdots a_n
\\ =&\Gamma_j(b,a_1)a_2\cdots a_n\pl.
\end{align*}
which coincides with $\Gamma_L(b,a_1a_2\cdots a_n)$. Then the induction step can be done using the product rule
\[\Gamma(xy,z)=y^*\Gamma(x,z)+\Gamma(y,x^*z)-\Gamma(y,x^*)z\pl.\]
That completes the proof.
\end{proof}
The above discussion naturally extends to free product of $n$ algebras. %In particular, we have the following corollary.

%\begin{cor}Let $T_{t}=e^{-At}:\M\to \M$ be a symmetric Markov semigroup with fixed point algebra $\N$. Let $\R$ be a finite von Neumann algebra containing $\N$. Suppose $(\A,\hat{\M},\delta)$ is a derivation triple of $T_t$. Define the derivation
%$\nabla:
%\A*_\N \R \to L_2(\hat{\M}_1*_\N\R)$
%\[\nabla(a_1a_2\cdots a_n)=\sum_{k \pl s.t.\pl  a_k\in \o{\M}} a_1\cdots \delta(a_{k})\cdots a_n \pl\]
%Then $(\A*_\N\R, \hat{\M}*_\N\R, \nabla)$ is a derivation triple for $T_{t}*\id_\R$.
%\end{cor}
%\begin{proof}Note that for the trivial semigroup $S_t=\id:\R\to \R$, its generator and derivation
%\[B(b)=0\pl, \delta_2(b)=0\pl, b\in \R\pl.\]
%vanishes. The assertion follows from applying Proposition \ref{freede} for $T_t$ and $S_T=\id_\R$.
%\end{proof}

\begin{prop}\label{freealg}Let $T_t:\M_1\to \M_1$ and $S_t:\M_2\to \M_2$ be two symmetric Markov semigroup with same fixed-point subalgebra respectively $\N_1$ and $\N_2$. Then\begin{enumerate}
\item[i)] if
$T_{1,t}$ and $T_{2,t}$ satisfies $\text{GRic}\ge \la$,
 $T_{1,t}*T_{2,t}$ on $\M_1*_\N\M_2$ satisfies $\text{GRic}\ge \la$.
 \item[ii)]if $T_{1,t}$ and $T_{2,t}$ satisfies $\la$-$\text{GRic}$,
 $T_{1,t}*T_{2,t}$ satisfies $\la$-$\text{GRic}$.
\end{enumerate}
 In particular, if $T_t$ satisfies $\text{GRic}\ge \la$ (resp. $\la$-\text{GRic}), then $T_t*\id$ satisfies $\text{GRic}\ge \la$ (resp. $\la$-\text{GRic}).
\end{prop}
\begin{proof}Let $(\A_1,\hat{\M}_1,\delta_1)$ (resp. $(\A_2,\hat{\M}_2,\delta_2)$) be a derivation triple of $T_{1,t}$ (resp. $T_{2,t}$) and $\hat{T}_{1,t}:\hat{\M}_1\to \hat{\M}_1$ (resp.  $\hat{T}_{2,t}:\hat{\M}_2\to \hat{\M}_2$) be the semigroup gives the $\Ric\ge \la $ relation of $T_{1,t}$ (resp. $T_{2,t}$). Namely, $\hat{T}_t|_{\M_1}=T_t$, $\hat{S}_t|_{\M_2}=S_t$ and
\[ \hat{A}\delta_1(x)-\delta_1 A(x)= \ric_1(\delta_1(x))\pl, \hat{B}\delta_2(x)-\delta_2B(x)= \ric_2(\delta_2(x))\pl.\]
with bimodule Ricci operator $\ric_A\ge \la$ and $\ric_B\ge \la$. Here $A,B$ (resp. $\hat{A},\hat{B}$) are the generator of $T_t,S_t$ (resp, $\hat{T}_t,\hat{S}_t$).
Consider the free product semigroup $\hat{T}_t=\hat{T}_{1,t}*\hat{T}_{2,t}: \hat{\M}_1*_\N\hat{\M_2}\to \hat{\M}_1*_\N\hat{\M_2}$. It follows that $\hat{T}_{1,t}*\hat{T}_{2,t}|_{\M_1*_\N\M_2}=T_{1,t}*T_{2,t}$. The generator of $\hat{T}_{1,t}*\hat{T}_{2,t}$ is
 \begin{align*}
 &\hat{L}a=\hat{A}a\pl, \hat{L}b=\hat{B}b\pl, a\in \hat{\M}_1, b\in \hat{\M}_2
 \\ &\hat{L}(a_1\cdots a_n)=\sum_{k}a_1\cdots \hat{L}a_k\cdots a_n
\end{align*}
where $a_k\in {\dom(\hat{A}_{i_k})}\cap \o{\M}_{i_k}, i_1\neq i_2\neq\cdots \neq i_n$ and similarly for the generator $L$ of $T_{1,t}*T_{2,t}$.
Since $\hat{L}|_{\M_1*_\N\M_2}=L$, for $a_k\in \o{\A}_{i_k}$,
\begin{align*}
&\hat{L}\delta(a_1\cdots a_n)-\delta L(a_1\cdots a_n)\\=& \sum_{1\le k\neq l\le n}^n \Big (a_1\cdots \hat{A}_{i_k}a_k\cdots \delta_{i_l}(a_l) \cdots a_n-a_1\cdots A_{i_k}a_k\cdots \delta_{i_l}(a_l) \cdots a_n\Big)\\ &+ \sum_{1\le k\le n}^n \Big (a_1\cdots \hat{A}_{i_k}\delta_{i_k}(a_l) \cdots a_n-a_1\cdots  \delta_{i_k}(A_{i_k}a_k) \cdots a_n\Big)
\\=& \sum_{1\le k\le n}^n a_1\cdots \ric_{i_k}(\delta_{i_k}(a_k)) \cdots a_n\pl.\end{align*}
Note that $\Omega_\delta=\overline{(\A_1*_\N\A_2)\delta(\A_1*_\N\A_2)}\subset L_2(\M_1*_\N\M_2)$ and by Leibniz rule
\begin{align}(\A_1*_\N\A_2)\delta(\A_1*_\N\A_2)=\bigoplus_{n\ge 1, 1\le k\le n}\bigoplus_{i_1\neq i_2\neq\cdots \neq i_n} \o{\A}_{i_1}\o{\A}_{i_2}\cdots\o{\A}_{i_k}\delta(\A_{i_k})
\cdots\o{\A}_{i_n}\pl. \label{sum}\end{align}
Moreover, the above decomposition are mutually orthogonal.
Now we define $\ric:\Omega_\delta\to L_2(\hat{\M}_1*_\N\hat{\M}_2)$ as
\[\ric(a_1\cdots a_k\delta_{i_k}(b_k) \cdots a_n)=a_1\cdots \ric_{i_k}(a_k\delta_{i_k}(a_k)) \cdots a_n\pl.\]
which is clearly a $\A_1*_\N\A_2$-bimodule operator. Now let us focus on a vector $h=\sum_{j=1}^m\eta_j\xi_j\gamma_j\in \o{\A}_{i_1}\o{\A}_{i_2}\cdots\o{\A}_{i_k}\delta(\A_{i_k})
\cdots\o{\A}_{i_n}$ with
\[\eta_k\in \o{\A}_{i_1}\o{\A}_{i_2}\cdots\o{\A}_{i_{k-1}}\pl, \xi_k\in \o{\A}_{i_k}\delta(\A_{i_k})\pl, \gamma_k\in \o{\A}_{i_{k+1}}\cdots\o{\A}_{i_n}\pl.\]
Then
\begin{align*}
\norm{h}{2}^2=&\sum_{j,l=1}^m\tau\Big( \gamma_j^*\xi_j^*\eta_j^*\eta_l\xi_l\gamma_l\Big)
\\=&\sum_{j,l}\tau\Big( E_\N (\gamma_l\gamma_j^*)\xi_j^*E_\N(\eta_j^*\eta_l)\xi_l\Big)
\end{align*}
Denote
\begin{align*} &X=\sum_{j,l=1}^me_{j,l}\ten E_\N(\eta_j^*\eta_l) \in M_m(\N)\pl, Y=\sum_{j,l=1}^me_{l,j}\ten E_\N (\gamma_l\gamma_j^*) \in M_m(\N)\pl, \\ &Z=\sum_{j=1}^me_{j,j}\ten \xi_j\in M_n(\hat{\M}_{i_k}). \end{align*}
We have
\begin{align*}
\norm{h}{2}^2=\tr\ten \tau(YZ^*XZ)=\norm{ X^{1/2}ZY^{1/2}}{2}
\end{align*}
Since $\ric$ is $\A_1*_\N\A_2$-bimodule operator and $\N\subset\A_1*_\N\A_2$, we have
\begin{align*}
\lan h, \ric(h)\ran&=\sum_{j,l}\tau\Big( E_\N (\gamma_l\gamma_j^*)\xi_j^*E_\N(\eta_j^*\eta_l)\ric(\xi_l)\Big)
\\ &=\tr\ten \tau(YZ^*X\id\ten\ric(Z))
\\ &=\lan X^{1/2}ZY^{1/2}, \id\ten\ric(X^{1/2}ZY^{1/2})\ran\ge \la \norm{ X^{1/2}ZY^{1/2}}{2}=\norm{h}{2}^2
\end{align*}
It then follows from orthogonality that $\ric\ge \la$ as an operator on $\Omega_\delta\subset L_2(\hat{\M}_1*_\N\hat{\M}_2)$. That completes the proof.
\end{proof}

\section{Optimal curvature for word length semigroups}
In this section, we discuss the optimal $\ARic$ conditions for several word length semigroups. This implies the corresponding CGE and CLSI results that are also independently obtained in \cite{WZ}.
\subsection{The $q$-Gaussian algerbas}
The $q$-deformed  Gaussian variable
was introduced by Frisch and Bourret in \cite{FB}. We refer to \cite{BS,BS2} for the more information about $q$-Gaussian operator algebra.
Let $H$ be a separable real Hilbert space and $H_\C$ be its complexification. Let $F(H)=\Omega\oplus (\oplus_{n\ge 1} H_\C^{\ten n})$ be the algebraic Fock space over $H_\C$ where $\Omega$ is the distinguished unit vector for the vacuum state. Let $-1\le q\le 1$. We equipped $\F(H)$ with $q$-deformed sesquilinear form,
\begin{align}\label{qform} \displaystyle\lan h_1\ten \cdots \ten h_n, k_1\ten \cdots \ten k_m\ran_q=\delta_{n,m}\sum_{\si\in S_n} q^{\iota(\si)}\Pi_{j=1}^n\lan h_j,k_{\si(j)}\ran\pl, \pl h_j,k_j\in {H_\C}\end{align}
Here $S_n$ denotes the permutation group on $n$ characters and $\iota(\si)$ denotes the inversion number of $\si\in S_n$.
This form is nonnegative definite and strictly positive definite for $-1<q<1$. Denote $\F_q(H)$ be the Hilbert space completion of $\F(H)$ with respect to $\lan\cdot ,\cdot \ran_q$. Define the left creation operator that for $h\in H_\C$,
\begin{align}\label{an} l_q(h) h_1\ten \cdots\ten h_n =h_0 \ten h_1\ten \cdots\ten h_n\pl.\end{align}
Its adjoint is the left annihilation operator
\[l_q^*(h) h_1\ten \cdots\ten h_n=\sum_{j=1}^n q^{j-1}\lan h, h_j \ran h_1\ten \cdots \ten \o{h}_j \ten \cdots h_n\pl,\]
where $\o{h}_j$ means the $j$ component is missing. For $-1\le q<1$, $l_q(h)\in B(\F_q(H))$ and satisfy the $q$-commutation relation.
\[l_q(h_1)l_q(h_2)^*-ql_q(h_1)^*l_q(h_2)=\lan h_1,h_2\ran \cdot 1\pl.\]
Let $s_q(h)=l_q(h)+l_q(h)^*$. The $q$-Gaussian von Neumann algebra is defined as
\[\Gamma_q(H):=\{ s_q(h)| h\in H\}''\subset B(\F_q(H))\]
%Let $\{e_j| j\in \mathbb{N}_+\}$ be a ONB of $H$.
%$s(e_1),\cdots, s(e_d)$ are generators of $\Gamma_q(H)$.
For ease of notation, we will suppress the ``$q$'' in the generator $s_q(h)$.
The canonical trace is given by the vacuum state, $\tau(x)=\lan \Omega, x \Omega\ran_q $ where $\Omega$ is the vacuum vector in $F(H)$. The distribution of $q$-Gaussian variables is given by the following formula \cite{BS} that for $h_1,\cdots, h_n$,
\begin{align*}
  \tau(s(h_1)\cdots s(h_n))=\begin{cases}
                         \sum_{\si\in P_2(n)}q^{c(\si)}\prod\limits_{\{i,j\}\in\si}\lan h_i,h_j\ran, & \mbox{if } n\pl \text{even} \\
                          0    , & \mbox{if } n\pl \text{odd}
                            \end{cases}
\end{align*}
where $P_2(n)$ denotes the pair partition of the set $\{1,\cdots, n\}$ and $c(\si)$ is the crossing number of the partition $\si$. For each $\xi\in \F(H)$, the Wick word of $\xi$ is the unique element $W(\xi)\in \Gamma_q(H)$ such that $W(\xi)\Omega=\xi$. Given a contraction $T:H\to K$, it induces a quantization contraction \[\F(T): \F_q(H)\to \F_q(H)\pl ,\pl  F(T)(h_1\ten \cdots \ten h_n)=Th_1\ten \cdots \ten Th_n\pl.\]
and a normal completely positive unital map \[\Gamma(T): \Gamma_q(T)\to \Gamma_q(T)\pl, \pl \Gamma(T)W(h_1\ten \cdots \ten h_n)=W(Th_1\ten \cdots \ten Th_n)\pl.\]
Moreover,
\begin{align*} \F(T^*)=\F(T)^*,  \F(ST)=\F(S)\F(T)\pl,  \Gamma(S^*)=\Gamma(S)^\dag \Gamma(ST)=\Gamma(S)\Gamma(T)\end{align*}
where $\Gamma(S)^\dag$ is the adjoint map of $\Gamma(S)$ with respect to trace inner product.
If $T$ is an isometry, $\Gamma(T)$ is an injective *-homomorphism. If $T$ is self-adjoint, $\Gamma(T)$ is symmetric.

Let $t\ge 0$ and $V_t: H\to H$ be the contraction \[V_t(h)=e^{-t}h\pl, \pl h\in H\pl.\]
This induces the word length semigroup $T_t: \Gamma_q(H)\to \Gamma_q(H)$ that
\[T_t(W(\xi))=e^{-mt}W(\xi) \pl, \pl \xi\in H^{\ten m}\pl.\]
$T_t$ is an ergodic symmetric quantum Markov semigroup.
Denote $E_m$ as the space of Wick words of length $m$.
Note $E_m$ are mutually orthogonal subspace of $L_2(\Gamma_q(H\ten K))$.
The generator of $T_t$ is the number operator
\[ NW(\xi)= m W(\xi) \pl, \pl \xi\in E_m\pl.\]
The Dirichlet algebra is then $\A_\E=\{W(\xi)| \norm{N^{1/2}W(\xi)}{2}<\infty\}$.
Let $\A_q(H)=\{ W(\xi)\in \F(H)\}$ be the $*$-algebra of Wick words of finite length. $\A_q(H)$ is clearly a w$^*$-dense subalgebra of $\Gamma_q(H)$ and a norm-dense subalgebra of $\A_\E$ with respect to the graph norm $\norm{x}{\E}=\norm{x}{2}+\norm{N^{1/2}x}{2}$.
The gradient form of $T_t$ is that for $\xi\in E_m,\eta\in E_{n}$
\begin{align*} 2\Gamma(W(\xi), W(\eta))&=W(\xi)(NW(\eta))+(NW(\xi))^*W(\eta)-N(W(\xi)^*W(\eta))\\ &=\sum_{|n-m|\le l\le n+m}(n+m-l)P_l(W(\xi)^*W(\eta))\pl.\end{align*}
where $P_l:L_2(\Gamma_q(H))\to E_l$ is the projection onto Wick words of length $l$.

\begin{lemma}\label{qderi} Define the map
\[\delta:\A_q(H)\to \A_q(H\oplus H)\pl, \pl\delta(W(h_1\ten \cdots \ten h_m))=\sum_{j=1}^mW(h_1\ten \cdots \ten \hat{h}_j\ten\cdots\ten h_m)\pl\pl.\]
where $h_j\in H$ and $\hat{h}_j\in 0\oplus H$ is the vector corresponding to $h_j$.
Then \begin{enumerate}
\item[i)] $\delta$ is a $*$-preserving closable derivation such that $E(\delta(x))=0$ for any $x\in \A_q(H)$. Here $E: \Gamma(H\oplus H)\to \Gamma(H)$ is the conditional expectation induced by the projection $H\oplus H\to H\oplus 0$.
\item[ii)]$(\A_q(H),\Gamma_q(H\oplus H),\delta)$ is a derivation triple for word length semigroup $T_t$.
    \end{enumerate}
\end{lemma}
\begin{proof}
We will repeatedly use the relation
\[ s(h_0)W(h_1\ten \cdots \ten h_m)=W(h_0\ten h_1\ten \cdots \ten h_m)+\sum_{j=1}^m\lan h_0,h_j \ran q^{j-1} W(h_1\ten \cdots\ten \overset{\circ}{h}_j \ten\cdots \ten h_m)\pl.\]
where $\overset{\circ}{h}_j$ means the $j$-th component is missing.
Since the Wick words are polynomials of the generators $s(h)=W(h)$, it suffices to verify the Leibniz rule that for $h_0,\cdots,h_m\in H$
\[\delta(W(h_0)W(h_1\ten \cdots \ten h_m))=\delta(W(h_0))W(h_1\ten \cdots \ten h_m)+W(h_0)\delta(W(h_1\ten \cdots \ten h_m))\]
We prove this by induction. First, note that $W(h)=s(h)$  and
\begin{align*}&W(h_0)W(h_1)=s(h_0)s(h_1)=W(h_0\ten h_1)-\lan h_0,h_1\ran\pl,
\\ &W(\hat{h}_0)W(h_1)=W(\hat{h}_0\ten h_1)\pl, \pl W({h}_0)W(\hat{h}_1)=W({h}_0\ten \hat{h}_1)\pl.
\end{align*}
Then the Leibniz rule is satisfied for $m=1$,
\begin{align*}
\delta(W(h_0)W(h_1))=&\delta(W(h_0\ten h_1)+\lan h_0,h_1\ran)=W(\hat{h}_0\ten h_1)+ W({h}_0\ten \hat{h}_1)\\ =&W(\hat{h}_0)W(h_1)+W({h}_0)W(\hat{h}_1)=\delta(W(h_0))W(h_1)+W(h_0)\delta(W(h_1))
\end{align*}
Assume that the Leibniz rule is satisfied for  all $m\le n$. By induction,
\begin{align*}
&\delta(s(h_0)W(h_1\ten \cdots \ten h_n))\\=&\delta\Big(W(h_0\ten h_1\ten \cdots \ten h_n)+\sum_{1\le j\le n} \lan h_0,h_j \ran q^{j-1}
W( h_1\ten \cdots \ten\mathring{h}_j\ten \cdots \ten h_n)\Big)
\\=&\sum_{0\le k\le n}W(h_0\ten h_1\ten \cdots \ten\hat{h}_k\ten\cdots\ten h_n)\\&+\sum_{1\le j\le n} \lan h_0,h_j \ran q^{j-1} \sum_{1\le k\le n, k\neq j}
W( h_1\ten \cdots \mathring{h}_j\cdots \ten\hat{h}_k\ten\cdots \ten h_n)
\\=& s(\hat{h}_0)W(h_1\ten \cdots \ten h_n)+s(h_0) \sum_{1\le k\le n}W(h_1\ten \cdots \ten\hat{h}_k\ten\cdots\ten h_n)
\\=& \delta(s(h_0))W(h_1\ten \cdots \ten h_n)+s(h_0) \delta(W(h_1\ten \cdots \ten h_n))\pl.
\end{align*}
This verifies $\delta$ is a derivation. To verify the form $\Gamma(x,y)=E(\delta(x)^*\delta(y))$, we first consider $x=W(h_0)$ and $y=W(\xi)=W(h_1\ten \cdots\ten h_m)$.
We denote \begin{align*}\o{\xi}_{j}=h_1\ten \cdots\ten \overset{\circ}{h}_j \ten\cdots\ten h_m \pl, \pl \hat{\xi}_{j}=h_1\ten \cdots\ten \hat{h}_j \ten\cdots\ten h_m \pl,
\end{align*}
where $\hat{h}$ is the copy of $h$ in $0\oplus H$. Thus,
\begin{align*}
E(\delta(s(h_0))^*\delta(W(\xi)))
=\sum_{j=1}^n E(s(\hat{h}_0)W(\hat{\xi}_j))
=\sum_{j=1}^n \lan h_0,h_j\ran q^{j-1} W(\o{\xi}_j)
\end{align*}
On the other hand,
\begin{align*}\Gamma(s(h_0),W(\xi))=&\frac{1}{2}\sum_{j=n-1, n+1} (n+1-j) P_j(s(h_0)W(\xi_j))\\
=&\frac{1}{2}\sum_{j=n-1, n+1} (n+1-j) P_l(s(h_0)W(\xi_j))\\
=& \sum_{l=1}^n \lan h_0,h_j\ran q^{j-1} W(\o{\xi}_j)=E(\delta(s(h_0))^*\delta(W(\xi)))\pl.
\end{align*}
which verifes the case $x=s(h_0), y=W(\xi)$. Note that both $\Gamma(x,y)$ and $E(\delta(x)^*\delta(y))$ satisfies product the rule
\[\Gamma(xy,z)=y^*\Gamma(x,z)+\Gamma(y,x^*z)-\Gamma(y,x^*)z\pl.\] Then by induction, the desired equality holds for products of $s(h_0)$ which spans all Wick words of finite length.
%This also implies that $\delta$ admits closed extension $bar{\delta}:\dom(N^{1/2})$ such that $\delta^*\bar{\delta}=N$.
Finally, the mean zero property follows from $E(\delta(W(\xi))=E(W(\hat{\xi}))=W(P\hat{\xi})=0$.
That completes the proof.
\end{proof}

\begin{theorem}\label{qgaussian}Let $H$ be a real separable Hilbert space. Let $T_t$
be word length semigroup on $q$-Gaussian algebra $\Gamma_q(H)$. Then $T_t$ satisfies optimal $1$-$\Ric$. As a consequence, $T_t$ satisfies optimal $1$-CGE and $1$-CLSI.
\end{theorem}
\begin{proof}Consider the Hilbert space contraction
\[\hat{O}_t:H\oplus H \to H\oplus H\pl, \hat{O}_t(h_1\oplus h_2)=e^{-t}h_1\oplus h_2 \]
Let $\hat{T}_t:\Gamma_q(H\oplus H)\to \Gamma_q(H\oplus H)$ be the quantization $\hat{O}_t$ as a symmetric quantum Markov semigroup. For $h_1\ten\cdots\ten h_m\in E_m$, we have
\begin{align*}\delta\circ T_t(W(h_1\ten\cdots\ten h_m)) =&e^{-mt}\delta(W(h_1\ten\cdots\ten h_m))
\\=& e^{-mt}\sum_{k=1}^m W(h_1\ten\cdots\hat{h}_k\cdots\ten h_m)
\\=& e^{-t}\hat{T}_t(\sum_{k=1}^m W(h_1\ten\cdots\hat{h}_k\cdots\ten h_m))
\\=& e^{-t}\hat{T}_t(\delta(W(h_1\ten\cdots\ten h_m)))\pl.\end{align*}
This verifies $T_t\circ \delta=e^{-t}\hat{T}_t\circ \delta$ on $\A$. The assertion follows from Proposition \ref{alg}. The constant is optimal because the spectral gap of the number operator $N$ is also $1$.
\end{proof}

\subsection{CAR algebras revisited}The CAR algebras (also called Clifford algebra) can be viewed as special case of $q$-Gaussian algebras for $q={-1}$. The word length semigroup are the special case of Orstein-Unlenbeck semigroup at temperature zero, which are studied in \cite{CM18,CM12} by Carlen and Maas for entropy Ricci curvature lower bound. The $\ARic$ result in this section is essentially due to discussion in \cite[Section 8]{CM18}. Because
their result will be used in the next example, we briefly revisit their results in our setting.

%We equipped $\F(H)$ with the anti-symmetric sesquilinear form,
%\begin{align*} \lan h_1\ten \cdots \ten h_n, k_1\ten \cdots \ten k_m\ran=\delta_{n,m}\Sigma_{\si\in S_n} (-1)^{\iota(\si)}\Pi_{j=1}^n\lan h_j,k_{\si(j)}\ran\pl, \pl h_j,k_j\in {H_\C}\end{align*}
%This form is non-negative and degenerate. The anti-symmetric $\F_a(H)$ is defined as the Hilbert space completion of the quotient by the corresponding kernel.
%Denote $c(h):=l_{a}(h)+l_{a}(h)^*$ where $l_a(h)$ is the annihilation operator for $q={-1}$ defined as in \eqref{an}. The Clifford algebra associated to $H$ is defined

Let $H$ be a real separable Hilbert space and $\{e_j\}$ be a ONB of $H$. The Clifford algebra $cl(H)$ is the unital $*$-algebra generated by a family of generators $\{c_j\pl | \pl 1\le j \le \text{dim} H\}$ satisfying the anti-commutation (CAR) relation
\[ c_ic_j+c_jc_i=2\delta_{ij}1\pl, c_i=c_i^*\pl. \]
Denote $[H]=\{1,2,\cdots ,\dim(H)\}$ if $\dim(H)<\infty$ and $[H]=\mathbb{N}_+$ for $\dim(H)=\infty$. For any finite subset $A\subset [H]$, we define the monomial
\[c_A=c_{j_1}c_{j_2}\cdots c_{j_m}\pl, A=\{j_1,\cdots, j_m\}\]
The canonical trace on $cl(H)$ is given by
\[\tau(c_A)=\begin{cases}
              1, & \mbox{if } A=\emptyset \\
              0, & \mbox{otherwise}.
            \end{cases}\]
Define $Cl(H)$ as the von Neumann algebra generated by the GNS representation. It is well-known that
$Cl(\mathbb{R}^d)$ is $2^n$-dimensional $C^*$-algebra and for $\dim(H)=\infty$, $Cl(H)=\R$ is the hyperfinite II$_1$ factor.
In particular $\{c_A\pl |\pl A\subset [H]\}$ forms a orthonormal basis for $L_2(Cl(H),\tau)$. Define the number operator
\[ Nc_A=|A|c_{A}\pl. \]
The word length semigroup $T_t=e^{-Nt}:Cl(H)\to Cl(H)$ is
\[ T_t(c_A)=e^{-|A|t}c_{A}\pl.\]
$T_t$ is an ergodic symmetric quantum Markov semigroup. The following theorem (for finite dimensional $H$) essentially follows from the argument for \cite[Theorem 8.6]{CM18} by Carlen and Maas. For separable $H$, the proof follows identically as in Lemma \ref{qderi} and Theorem \ref{qgaussian} for $q$-Gaussian.

\begin{theorem}\label{clifford}Let $H$ be a real separable Hilbert space. Let $T_t$
be word length semigroup on the Clifford alegbra $Cl(H)$. Define the map
\[\delta:cl(H)\to cl(H\oplus H)\pl, \pl\delta(c_{j_1}\cdots c_{j_m})=\sum_{k=1}^mc_{j_1}\cdots \hat{c}_{j_k}\cdots c_{j_m}\pl\pl.\]
where $\hat{c}_{j_k}\in cl(0\oplus H)\subset cl(H\oplus H)$ corresponding to $c_{j_k}\in cl(H)$.
\begin{enumerate}
\item[i)]$\delta$ is a symmetric closable derivation with mean zero property and $(cl(H),Cl(H\oplus H),\delta)$ is a derivation triple for $T_t$.
\item[ii)] $T_t$ satisfies optimal $1$-$\Ric$, $1$-CGE and $1$-CLSI.
\end{enumerate}
\end{theorem}
\begin{exam}[Two point space]{\label{twop}\rm
For $\dim(H)=1$, $Cl(\mathbb{R})=\mathbb{C}\oplus \mathbb{C}$ is the commutative $C^*$-algebra of two point space. Write $\bm{1}=(1,1)$ and $\epsilon=(1,-1)$ as a basis for $Cl(\mathbb{R})$. The word length semigroup is
\[T_t(a\bm{1}+b\epsilon)=a\bm{1}+e^{-t} b\epsilon \pl , \pl \forall\pl  a,b \in \mathbb{C}\pl.\]
}
\end{exam}
\subsection{Free groups}
In this part we prove the optimal Ricci curvature and CLSI for word length semigroup on free group factor. Our idea is to use Theorem \ref{freealg} that $\Ric\ge\la$ condition is stable under free product. We denote $\Z_2$ as the order $2$ group, $\Z$ as the integer group, and $\mathbb{F}_d$ as the free group of $d$ generator. We write $G_N={\Z_2*\Z_2\cdots *\Z_2}$ as the free product of $N$-copies of $\mathbb{Z}_2$.
We start with a corollary of Example \eqref{twop} and Theorem \ref{freealg}.

\begin{prop}Let \[P_{G_N,t}:\L(G_N)\to \L(G_N), P_{G_N,t}(\la(g))=e^{-|g|t}\la(g)\] be the word length semigroup on $\L(G_N)$, where $g$ is a word of ${a_j=a_j^{-1},1\le j\le d}$ and $|g|$ is the word length of $g$. Then $P_{G_N,t}$ satisfies optimal $\text{GRic}\ge 1$ and $1$-CLSI.
\end{prop}
\begin{proof}Note that the $\L(G_d)=\L(\Z_2*\Z_2\cdots *\Z_2)=\L(\Z_2)*\cdots *\L(\Z_2)$ and the semigroup
\[P_{G_N,t}=P_{\Z_2,t}*P_{\Z_2,t}*\cdots *P_{\Z_2,t}\pl,\]
 where $P_{\Z_2,t}$ is the word length semigroup on the two-points space $\L(\Z_2)\cong \mathbb{C}\oplus \mathbb{C}$ discussed in Example \eqref{twop}. The assertion follows from Theorem \ref{freealg}.
\end{proof}
\begin{cor}\label{possion}Let \[P_{\Z,t}:\L(\Z)\to \L(\Z), P_{\Z,t}(\la(u))=e^{-|m|t}u^m\] be the word length semigroup on $\L(\Z)$, where $u=\la(\bm{1})$ is the unitary generator of $\L(\Z)$. Then $P_{\Z,t}$ satisfies optimal $\text{GRic}\ge 1$ and $1$-CLSI.
\end{cor}
\begin{proof}Recall the \cite[Lemma 3.3]{JPPPR} that
\[ \pi\circ P_{\Z,t}=(id_{M_2}\ten P_{\Z_2*\Z_2,t})\ten \pi\]
where $\pi$ is the $*$-homomorphism
\[\pi:L(\Z)\to M_2\ten L(\Z_2*\Z_2)\pl, \pi(u)=\left[\begin{array}{cc} 0 & \la(a_{1})\\ \la(a_{2})&0\end{array}\right]\]
Here $u$ is the generator of $\Z$ and $a_1$ (resp. $a_2$) is the nontrivial element in the first (resp. second) copy of $\Z_2$ in $\Z_2*\Z_2$. This implies $P_{\Z,t}$ is a transference semigroup of $P_{\Z_2*\Z_2,t}$. Hence $P_{\Z,t}$ satisfies $\text{GRic}\ge 1$ and $1$-CLSI. This is optimal because the spectral gap of $P_{\Z,t}$ is also $1$.
\end{proof}

\begin{rem}{\rm By $\L(\Z)\cong L_\infty(\T)$, $P_{\Z,t}$ is the Possion semigroup on the torus $\T$.}
\end{rem}

\begin{cor}Let \[P_{\FF_d,t}:\L(\FF_d)\to \L(\FF_d), P_{\Z,t}(\la(g))=e^{-|m|t}\la(g)^m\] be the word length semigroup on $\L(\FF_d)$, where $g$ is the free word of $d$ generators and $|g|$ is the word length of $g$. Then $P_{\FF_d,t}$ satisfies optimal $\text{GRic}\ge 1$ and $1$-CLSI.
\end{cor}
\begin{proof}Note that the $\L(\FF_d)=\L(\Z*\Z*\cdots *\Z)=\L(\Z)*\cdots *\L(\Z)$ and the semigroup
\[P_{\FF_d,t}=P_{\Z,t}*P_{\Z,t}*\cdots *P_{\Z,t}\pl.\]
Then by Corollary \ref{possion} and Theorem \ref{freealg}, $P_{\FF_d,t}$ satisfies $\text{GRic}\ge 1$ and $1$-CLSI. The optimality follows from that the spectral gap is $1$.
\end{proof}
\begin{rem}{\rm The best known LSI constant proved in \cite{JPPPR} is $(1+\frac{\log 2}{4})^{-1}<1$, which implies $(1+\frac{\log 2}{4})^{-1}$-MLSI. Here we obtained the sharp $1$-CLSI. The sharp $1$-LSI remains open.}
\end{rem}

\section{Other examples}
\subsection{Generalized depolarzing channel}\label{depo}
Let $\N\subset \M$ be a subalgebra and $E:\M\to \N$ be the conditional expectation. Define the generalized depolarzing channel
\[T_t(\rho)=e^{-\la t}\rho+(1-e^{-t})E(\rho)\]
It was proved in \cite{BGJ} and \cite{WZ} that $T_t$ satisfies $1/2$-CGE. Here we use free product to show that $T_t$ satisfies the stronger condition $\ARic\ge 1/2$.
\begin{prop}\label{depolarizing}The generalized depolarzing channel $T_t(\rho)=e^{-\la t}\rho+(1-e^{-t})E(\rho)$
 satisfies $\ARic\ge 1/2$.
\end{prop}
\begin{proof}
Let $\M_1=\M$ and $\M_2=\L(\Z_2)\ten\N$. Consider the semigroup $P_t=P_{\Z_2,t}\ten \id_\N$ on $\L(\Z_2)\ten\N$ where
\[P_{\Z_2,t}(a\bm{1}+b\epsilon)=a\bm{1}+e^{-t}b\epsilon \]
is the semigroup on $\Z_2$. Now consider the embedding
\[\pi:\M\to \M_1*_\N\M_2\pl, \pi(x)=\epsilon x \epsilon \pl.\]
$\pi$ is an injective trace-preserving $*$-homomorphism because $\epsilon$ is uniatry. Note that for $x=E(x)+\o{x}$, $\pi(x)=\epsilon x \epsilon=\epsilon E(x)\epsilon+\epsilon\o{x}\epsilon= E(x)$ because
$\L(\Z_2)$ and $\N$ commute in $\M_1*_\N\M_2$. Then
we have the interwining relation
\[ P_t*\id_\M(\pi(x))=P_t*\id_\M(E(x)+\epsilon\o{x}\epsilon)=E(x)+e^{-2t}\epsilon\o{x}\epsilon=\pi(T_{2t}(x))\pl.\]
Note that $P_{\Z_2,t}$ and hence $P_t*\id_\M=(P_{\Z_2,t}\ten \id_\N)* \id_\M$ has $\Ric\ge 1$. Then $T_{2t}$ as a subsystem of $P_t*\id_\M$ has $\Ric\ge 1$. Therefore $T_{t}$ has $\Ric\ge 1/2$.
\end{proof}

\begin{prop} The above constant $1/2$ is sharp for general conditional expectation $E$.
\end{prop}
\begin{proof}Note that $\ARic\ge \la$ implies
 $\Gamma_2\ge \la \Gamma$ . We show that $\Gamma_2\ge \la \Gamma$ is optimal for a general conditional expectation.
 Let $G$ be a discrete group and $\L(G)$ be its group von Neumann algebra. Consider the Fourier multiplier discussed in Section \ref{f},
 \[T_t(\la(g))=e^{-\psi(g)t}\la(g)\pl, A(\la(g))=\psi(g)\la(g)\pl.\]
 The gradient form and $\Gamma_2$ operator are
 \[\Gamma(\la(g),\la(h))=K(g,h)\la(g)^*\la(h)\pl, \Gamma(\la(g),\la(h))=K(g,h)^2\la(g)^*\la(h)\pl,\]
where $2K(g,h)=\psi(g)+\psi(h)-\psi(g^{-1}h)$. Thus $\ARic\ge \la$ implies $\Gamma_2\ge \la \Gamma$, which is \[[K(g,h)^2]_{g,h\in G}\ge \la [K(g,h)]_{g,h\in G}\]
as matrices on $l_2(G)$.
Indeed, take $\M=\L(\Z)$ and $\N=\mathbb{C}1$. The depolarzing semigroup has generator
\[A(\la(g))=\bm{1}_{g\neq 0}\la(g)\pl, g\in \Z\pl.\]
which corresponds to the Fourier multiplier of indicator function $\psi(g)=\bm{1}_{g\neq 0}$. Then
\[K(g,h)=\frac{1}{2}\bm{1}_{h\neq 0}\bm{1}_{g\neq 0}(1+\bm{1}_{g\neq h})=\frac{1}{2}\bm{1}_{gh\neq 0}+\frac{1}{2}\bm{1}_{g=h\neq 0} \pl.\]
and
\[K(g,h)^2=\frac{1}{4}\bm{1}_{gh\neq 0}+\frac{3}{4}\bm{1}_{g=h\neq 0}\ge \frac{1}{4}\bm{1}_{gh\neq 0}+\frac{1}{4}\bm{1}_{g=h\neq 0}=\frac{1}{2}K(g,h)\pl.\]
Therefore $1/2K^2\ge K$ is the best possible constant by checking vector $v=\sum_{j=1}^N\ket{j}\in l_2(\Z)$ for arbitrary large $N$.
\end{proof}

\begin{cor}Let $\M_1,\cdots,\M_n$ be $n$ finite von Neumann algebras and $\N$ is a common subalgebra of $\M_j, 1\le j\le n$. Then the word length semigroup $P_t : \M_1*_\N\cdots*_\N\M_n\to\M_1*_\N\cdots*_\N\M_n$
\[ P_t(a_1\cdots a_m)=e^{-mt}a_1\cdots a_m\pl, a_i\in \o{\M}_{i_j}\pl, i_1\neq i_2\neq \cdots \neq i_m\pl \]
satisfies $\ARic\ge 1/2$.
\end{cor}
\begin{proof}Use Proposition \ref{depolarizing} and \ref{freealg}.
\end{proof}

\subsection{Quantum Tori}Quantum tori are prototype examples in noncommuative geometry.
Let $d\ge 2$ and $(\theta_{jk})_{j,k=1}^d$ be a $d\times d$ skew-symmetric real matrix. The $n$-dimensional quantum torus $\A_\theta$ is the universal $C^*$-algebra generated by $d$-tuple of unitaries $(u_1, u_2, \cdots, u_d)$ satisfying the commutation relation
\begin{align}u_ju_k=e^{2\pi i \theta_{jk}}u_ku_j\pl, \pl j ,k=1, 2,\cdots ,d,  \pl \label{atheta}
\end{align}
 We denote $u=(u_1,u_2, \cdots, u_d), m=(m_1,m_2,\cdots, m_d)$ and use the standard notation of multiple Fourier series as follows,
\[ u^m=u_1^{m_1}u_2^{m_2}\cdots u_d^{m_d}\pl. \]
The canonical tracial state is
 \[\tau(\sum_{m\in \Z^d}\al_m u^m)=\al_0 \pl.\]
 The monomials $\{u^m | m\in \Z^d\}$ forms a ONB of $L_2(\A_\theta,\tau)$.
Denote $\R_\theta$ be the von Neumann algebra as the $w^*$-closure of the GNS representation $\A_\theta\subset B(L_2(\A_\theta,\tau))$.
Let $\T^d=\{(z_1,z_2,\cdots, z_d) \in \C^d \pl |\pl |z_j|=1 \pl , \pl \forall  j\}$ be the $d$-torus. When $\theta=0$, $\A_\theta\cong C(\T^d)$ and $\R_\theta\cong L_\infty(\T^d,dm)$. The heat semigroup on $\R_\theta$ is defined as
\[T_t:\R_\theta \to \R_\theta\pl,\pl  T_t(u^m)=e^{-|m|^2t}u^m \pl,\]
where the generator is the Laplacian
\[\Delta(u^m)= |m|^2u^m\pl, \pl |m|=\sqrt{m_1^2+\cdots+m_n^2}\pl.\]
It is known that $T_t$ is the transference of heat semigroup on $\T^d$. This property has been used in \cite{CXY} for harmonic analysis on $\R_\theta$. For each $z=(z_1,z_2,\cdots, z_d)\in \T^d$, the associated transference action is given by
\[\al_z(u^m)=z^mu^m\equiv z_1^{m_1}z_2^{m_2}\cdots z_d^{m_d}u_1^{m_1}u_2^{m_2}\cdots u_d^{m_d} \pl.\]
Define the trace preserving $*$-monomorphism
\[\al:\R_\theta \to L_{\infty}(\T^d, R_\theta)\pl, \pl \al(x)(z)=\al_z(x)\pl.\]
Let $S_t:L_{\infty}(\T^d)\to L_{\infty}(\T^d), S_t(z^m)=e^{-|m|^2t}z^m$ be the heat semigroup on $\T^d$.
The following diagram commutes
\begin{equation}\label{ccss}
 \begin{array}{ccc}  L_{\infty}(\T^d, R_\theta)\pl\pl &\overset{ S_t\ten  \operatorname{id}_{R_\theta} }{\longrightarrow} & L_{\infty}(\T^d, R_\theta) \\
                    \uparrow \al    & & \uparrow \al  \\
                     R_\theta\pl\pl &\overset{T_t}{\longrightarrow} & R_\theta
                     \end{array} \pl .
                     \end{equation}
Namely, $(S_t\ten \id)\circ\al=\al \circ T_t$.
This means the semigroup $T_t$ is the restriction of $S_t\ten \id_{\R_\theta}$ on $\al(\R_\theta)$. Denote $\mathcal{P}_\theta=\text{span}\{u^m| m\in \mathbb{Z}^d\}$ as the algebra of polynomials. Then
$(\mathcal{P}_\theta, L_{\infty}(\T^d, R_\theta), \delta=(\nabla_{\T^d}\ten \id)\circ\al )$
gives a derivation triple of $T_t$, where $\nabla_{\T^d}: C^\infty(\T^d)\to T\T^d\cong \oplus_{j=1}^dC^\infty(\T^d)$ is the gradient operator on $\T^d$. More explicitly, \[\delta(u^m)(z)=(m_1\al_z(u^m),m_2\al_z(u^m),\cdots, m_d\al_z(u^m))\pl.\]
Conversely, consider the trace preserving $*$-homomorphism
\[ \al_\theta: L_\infty(\T^d)\to\R_\theta\bten \R_\theta^{op}\pl,\pl  \al_\theta(z^m)=u^m\ten (u^m)^{op}\]
where $\R_\theta^{op}$ is the opposite algebra of $\R_\theta$. $\al_\theta$ is a $*$-homomorphism because $\al_\theta(z_j)=u_j\ten u_j^{op},1\le j\le d$ are $d$-commuting unitaries. We have another commuting diagram
\begin{equation}\label{ccss}
 \begin{array}{ccc}  R_\theta\bten R_\theta^{op}\pl\pl &\overset{ T_t\ten  \operatorname{id}_{R_\theta^{op}} }{\longrightarrow} & R_\theta\bten R_\theta^{op}\\
                    \uparrow \al_\theta    & & \uparrow \al_\theta  \\
                     L_{\infty}(\T^d)\pl\pl &\overset{S_t}{\longrightarrow} & L_{\infty}(\T^d)
                     \end{array} \pl .
                     \end{equation}
Namely, $\al_\theta\circ S_t= (T_t\ten \id_{\R_\theta^{op}})\circ \al_\theta$. This means $S_t$ and $T_t$ are the transference semigroup of each other. By the completeness of the definition of $\ARic$, CGE and CLSI, we know that $T_t$ on $\R_\theta$ and $S_t$ on $\T_d$ have same $\ARic$, CGE and CLSI constant. Similar equivalence holds for other corresponding semigroup on $\R_\theta$ and $\T_d$.

\begin{cor}Let $\R_\theta$ be the $d$-dimensional quantum tori. Consider the quantum Markov semigroups \begin{align*}
&T_t:\R_\theta\to \R_\theta\pl , \pl T_t(u^m)=e^{-|m|^2t}u^m,\\
&P_t:\R_\theta\to \R_\theta \pl , \pl P_t(u^m)=e^{-|m|t}u^m,\\
&Q_t:\R_\theta\to \R_\theta \pl , \pl P_t(u^m)=e^{-\|m\|_{1}t}u^m,
\end{align*}
where $\|m\|_1=\sum_{j=1}^d|m_j|$. Then
\begin{enumerate}
\item[i)]$T_t$ has $\Ric\ge 0$ and $(4\ln 3)^{-1}$-CLSI.
\item[ii)] $P_t$ has $\Ric\ge 0$ and $(4\pi)^{-1} (\frac{2(d-1)!}{\Gamma(d/2)})^{-\frac{1}{d}}$-CLSI.
\item[iii)]$Q_t$ has optimal $\Ric\ge 1$ and $1$-CLSI.
\end{enumerate}
\end{cor}
\begin{proof}
By the transference trick above, it suffices to consider the corresponding semigroup on torus $\T_d$. i) follows from \cite[Theorem 4.12]{BGJ}. iii) follows from Lemma \ref{possion} and $Q_t$ is a tensor product semigroup of Possion semigroup on $\T$. ii) corresponds to the Possion semigroup on $\T_d$, which is also central. Then $P_t$ has $\Ric\ge 0$. The CB-return time estimate is
\begin{align*}
\norm{P_t-E_\tau:L_1(\mathbb{T}^d)\to L_\infty(\mathbb{T}^d)}{}&=\norm{\sum_{m\in \Z^d, m\neq 0}e^{-|m|t}z^mw^{-m}}{L_\infty(\mathbb T^d \times \mathbb T^d)}\\
&\le \sum_{m\in \Z^d, m\neq 0}e^{-|m|t}\\
&\le \int_{\mathbb{R}^d}e^{-|\bx|t} d\bx= s_d\int_{0}^\infty e^{-rt}r^{d-1} dr= \frac{s_d(d-1)!}{t^d}
\end{align*}
where $\displaystyle s_d=\frac{2\pi^d}{\Gamma(d/2)}$ is the surface are of $(d-1)$-dimensional unit sphere. The CLSI constant follows from Theorem \ref{CLSI3}.
\end{proof}

\bibliography{fdiv,asymmetry,nsfn,biblin22}
\bibliographystyle{plain}
\end{document}